\documentclass[10,reqno]{amsart}
\usepackage{amsfonts}
\usepackage{amsmath,amssymb,amsthm,bm}
\usepackage[active]{srcltx}
\usepackage{mathtools}
\usepackage{shuffle}
\usepackage{amssymb}
\usepackage{amsrefs}
\usepackage{empheq}
\usepackage{wasysym}
\usepackage{verbatim}
\usepackage{graphicx}
\usepackage{subfigure}
\usepackage[
bookmarks=true,         
bookmarksnumbered=true, 
colorlinks=true, pdfstartview=FitV, linkcolor=blue, citecolor=blue,
urlcolor=blue]{hyperref}
\usepackage[top=1.1in, bottom=1.2in, left=1.2in, right=1.2in]{geometry}
\usepackage{cancel}
\usepackage{mathrsfs}
\usepackage{fancyhdr}
\usepackage{amsthm}
\usepackage{cases}
\usepackage{enumitem}
\usepackage{xcolor}
\usepackage{upgreek}

\usepackage[leftcaption]{sidecap}
\usepackage{booktabs}
\usepackage{lipsum}
\usepackage{subfigure}
\usepackage{caption}
\usepackage{floatrow}
\newfloatcommand{capbtabbox}{table}[][\FBwidth]
\usepackage{threeparttable,booktabs}
\PassOptionsToPackage{normalem}{ulem}
\usepackage{ulem}

\graphicspath{ {./pics/} }

\providecolor{added}{rgb}{0,0,1}
\providecolor{deleted}{rgb}{1,0,0}

\def\RR{\mathbb{R}}

\def\EE{\mathbb{E}}

\def\cF{\mathcal{F}}

\def\cH{\mathcal{H}}

\def\cS{\mathcal{S}}

\def\DD{\mathbb{D}}
\def\ZZ{\mathbb{Z}}

\def\PP{{\mathbb{P}}}

\let\Om=\Omega

\let\al=\alpha

\def\Om{{\Omega}}
\def\al{{\alpha}}

\def\be{{\beta}}
\def\Ga{{\Gamma}}

\def\de{{\delta}}
\def\De{{\Delta}}

\def\la{{\lambda}}

\def\11{{\bf 1\mkern -7mu I}}

\def\de{{\delta}}

\def\al{{\alpha}}
\def\d{{\textrm{d}}}
\def\cP{\mathcal{P}}

\newcommand{\lcl}{\left\{}
\newcommand{\rcl}{\right\}}

\newcommand{\ep}{\varepsilon}

\newcommand{\lc}{\left(}
\newcommand{\rc}{\right)}

\def\De{{\Delta}}

\newtheorem{theorem}{Theorem}[section]
\newtheorem{lemma}[theorem]{Lemma}
\newtheorem{proposition}[theorem]{Proposition}
\newtheorem{assumption}[theorem]{Assumption}
\newtheorem{remark}[theorem]{Remark}

\newtheorem{prop}[theorem]{Proposition}
\newtheorem{definition}[theorem]{Definition}
\theoremstyle{definition}

\newtheorem{example}[theorem]{Example}

\numberwithin{equation}{section}

\begin{document}

\title{Backward Euler method for stochastic differential equations with non-Lipschitz coefficients}

\author{Hao Zhou}
\address{School of Mathematics, Harbin Institute of Technology, Harbin, 150001, China}
\email{haozhouhit@163.com}
\thanks{H. Zhou is supported by the China Scholarship Council.}

\author{Yaozhong Hu}
\address{Department of Mathematical and Statistical Sciences, University of Alberta, Edmonton, AB T6G 2G1, Canada}
\thanks{Y. Hu is supported by   an NSERC discovery grant   and a startup fund from   University of Alberta at Edmonton.}
\email{yaozhong@ualberta.ca}


\author{Yanghui Liu$^{\ast}$}
\address{Department of Mathematics, Baruch College, CUNY, New York, NY 10010, USA}
\email{yanghui.liu@baruch.cuny.edu}
\thanks{$^{\ast}$ Y. Liu is supported by the PSC-CUNY Award 64353-00 52.}




\maketitle

\begin{abstract}
We study  the traditional backward Euler method for  $m$-dimensional stochastic differential equations  driven by fractional Brownian motion with  Hurst parameter $H>1/2$ whose drift coefficient satisfies the one-sided Lipschitz condition. 
The  backward Euler scheme is
proved to be of order $1$ and this rate is optimal
by showing the asymptotic error distribution  
 result. Two numerical experiments are performed to validate our claims about the optimality of the rate of convergence.
\end{abstract}

\section{Introduction}

In this  paper  we consider the following $m$-dimensional stochastic differential equation (SDE for short):
\begin{align}\label{maineq}
\d X_t = b(X_t)\d t +  \d B_t,\quad \quad t\in(0,T],
\end{align}
where
$B_t$ is  a fractional Brownian motion  (fBm for short) with Hurst parameter $H>1/2$
 and
$X_0=x_0$ is a constant in $\RR^m$.
We assume that   the drift function  $b: \RR^{m}\to \RR^{m}$  satisfies the one-sided Lipschitz condition and that $b$ and its derivatives may be unbounded. One typical example that we have in mind is given by $b(x)=x-x^{3}$, $x\in\RR$ or any polynomial of odd degree with a negative leading coefficient.  The existence and uniqueness of  the solution for this equation 
 can be found    in e.g. \cite{hu2008singular, hairer2005ergodicity, boufoussi2005kramers}.
 The work \cite{hairer2005ergodicity} studied the ergodicity of the solution
for this type of equation.

  In this paper, we consider the numerical approximation of the equation.
The numerical approximation  of SDE  driven by Brownian motion is well studied; see e.g. the monographs  \cite{KP, M}.
The numerical approximation
for fractional SDE has also   received  much attentions in  recent works; see e.g. \cite{kloeden2011multilevel, mishura2008stochastic, riedel2020semi, hu2016rate, hong2020optimal}.
The  optimal approximation for the  additive noise SDE   is considered in   \cite{Neuenkirch2}.
The multilevel Monte Carlo method for additive SDE is demonstrated in
\cite{kloeden2011multilevel}.

The aim of this paper is  to study the numerical approximation of the SDE \eqref{maineq} under the   condition that the drift function $b$ is   one-sided Lipschitz only. It is well-known that for such SDE the explicit methods are  conditionally stable and    implicit methods are usually preferred;   see  Table \ref{table.new} in Example~\ref{example1} for  performance   comparison     among some numerical methods.
The numerical approximation for  such SDE has been investigated  in~\cite{riedel2020semi}, in which
the   almost sure convergence rate  $H$  of the backward Euler method is obtained.

In this paper, we focus on the backward Euler method. We  prove that the strong convergence  rate   for the backward Euler method is 1.
 The main step in our proof is to derive a decomposition  (see \eqref{eqn.zt} and  \eqref{eqn.nz}) of the error process in terms of the quantities of the form $(\text{Id}-\Delta \partial b   )^{-1}$ which is shown to be bounded by $e^{\lambda\Delta}$
 under the one-sided Lipschitz condition. Here $\partial b$ is   the Jacobian of  $b$, $\lambda$ the largest eigenvalue of $\partial b$, $\Delta>0$ some small number, and $\text{Id}$ is the identity matrix.
  This decomposition together with some Malliavin calculus techniques allow us to treat   the difficulties due to the unboundedness of the coefficient in the equation.
 In the second part of our paper, we  show  that  the asymptotic  error of the backward Euler method   converges in probability  to the solution of a linear SDE.  
    Our main approach in this part   is the combination of  a variation of parameter method   and   limit theorems for weighted random  sums (see \cite{hu2016rate, LT}).

The remainder of the paper is structured as follow.
Section \ref{Preliminaries} concerns some basic preparation work on  Malliavin calculus and Young integrals.
The properties of the exact solution and Malliavin derivatives of the SDE \eqref{maineq} are shown in Section \ref{well-posedness}.
In Section \ref{BEM}, the convergence rate of the backward Euler method is investigated.
The convergence rate is proved to be an optimal one in Section \ref{limit}.
Then in the Section \ref{Experiments}, we report    our numerical experiments and  illustrate the accuracy of the  method.

\section{Preliminaries}\label{Preliminaries}

\subsection{Elements of Malliavin calculus}

Let $(\Om, \cF, \PP)$ be a probability space  with a right continuous filtration
$(\cF_t, t\in [0, T])$ satisfying the usual condition. Let $(B_t , t\in [0, T])$ be the $m$-dimensional fBm of Hurst parameter $H=(H_1, \cdots, H_m)\in (0, 1)^m$. Recall that a  fBm
$(B_t =(B_t^1,  \cdots, B_t^m) , t\in [0, T])$ is a continuous
($m$-dimensional) mean zero  Gaussian process with   covariance given by
\begin{align*}
\EE (B_t^i B_s^j )=\frac{\de_{ij}}2 \left[ t^{2H}+s^{2H}-|t-s|^{2H}\right],\quad \forall \ s,  t\in [0, T] ,
\end{align*}
where $i,j=1,2 \cdots, m$ and $\de_{ij}$ is the Kronecker symbol.
 We shall assume that all the Hurst parameters are greater than $1/2$:    $H_i>1/2, i=1, \cdots, m$. The case $H_i<1/2$ demands  a different handling,  in particular for the central limiting type theorem, and will be considered in our future projects. For a set $A $ we denote   $\mathbf{1}_{A}$ the indicate function such that $\mathbf{1}_{A}(x)=1$ if $x\in A$ and  $\mathbf{1}_{A}(x)=0$ otherwise. Associated with this covariance, we define for $i=1, \cdots, m$,
\[
\langle \mathbf{1}_{(a, b]}\,, \mathbf{1}_{(c,d]}\rangle_{\cH_i}
=\int_a^b \int_c^d \phi_i(u,v) \d u\d v=
\frac{1}{2}\left[ |d-a|^{2H_i}+|c-b|^{2H_i}-|d-b|^{2H_i}-|c-a|^{2H_i}\right],
\]
where $0\le a<b\le T, 0\le c<d\le T$, and
\[
\phi_i(u,v)=H_{i}(2H_{i}-1)|u-v|^{2H_i-2}.
\]
$\langle \cdot,\cdot\rangle_{\cH_i}$ can be extended by (bi-)linearity to a scalar product on the set
$\cS $ of  step  functions of the form
$f(t)=\sum_{i=1}^n f_i\mathbf{1}_{(a_i, b_i]}$, $f_i\in \RR$.
We denote by $\cH_i$ the Hilbert space obtained by  completing  $\cS$ with respect to the above scalar product. Note that $\cH_i$  contains   generalized functions. If $f, g:[0, T] \rightarrow \RR$ are integrable measurable functions,  then
\[
\langle f, g\rangle_{\cH_i}=\int_0^T
\int_0^T f(u) g(v) \phi_i(u,v) \d u\d v.
\]
We denote $\cH=\cH_1\otimes \cdots\otimes \cH_m$ the Hilbert space of the tensor product of $\cH_1, \cdots, \cH_m$.

 For the stochastic analysis associated with the fractional Brownian motion, we refer to
 \cite{biagini2008stochastic, hu2005integral, mishura2008stochastic}
 and we shall use their  notations freely.
For example, we shall use
the It\^o type stochastic integral
$\de^i(f)=\int_0^T f(s) \d B_s^i$  and $\de^i(f\mathbf{1}_{(a, b]})=\int_a^b
f(s) \d B_s^i$ freely.
 But let us   recall some notations and results on the Malliavin calculus for the fractional Brownian motions since some of the results we are going to use cannot
 be found in the standard literature.  We denote by $\cP$ the set of all nonlinear functionals (random variables)  on $(\Om, \cF, \PP)$ of the form
\[
F= f(B_{t_1}, \cdots, B_{t_n}),
\]
where $0\le t_1<\cdots<t_n\le T$, and $f
=f((x_{ij})_{1\le i\le m, \,1\le  j\le n})$  is a smooth function of $n\times m$ variables of  polynomial growth. For $f\in \cP$ of the above form, we define its Malliavin derivative  as
\[
D^{(i )}_r F=\sum_{j=1}^n \frac{\partial }{\partial x_{ij}}  f(B_{t_1}\,, \cdots, B_{t_n})   \mathbf{1}_{[0, t_i] }(r),
\qquad
i=1, \cdots, m.
\]
We consider $D^{(i )}  F =D^{(i )} _\cdot F$  as an  $\cH_i $-valued random variable and consider $D   F=(D^{(1 )}F, \cdots, D^{(m )}F)$ as an $\cH$-valued random variable.  We can also define the higher order derivatives $D^{k}F$   which can be considered
as an $\cH^{\otimes k}$-valued random variable.
For $p>1$ we denote
\begin{align*}
\|F\|_{k, p}^p:=\sum_{\ell=0 }^{k}
\EE \left[ \|D^{\ell}F\|
_{\cH^{\otimes \ell}}^p
\right] ,
\end{align*}
where $\ZZ_+=\{0, 1,2, \cdots\}$ is the set of nonnegative integers
and we denote $\DD^{k, p}$ the completion of  $\cP$ with respect to the above norm.

We recall
\begin{equation}
\EE \left[ F\int_0^T
 g(s) \d B_s^i\right]
=  \EE\left[ \int_0^T \int_0^T
    D_s^i F g(t) \phi_i(s, t) \d s\d t
\right] =\EE
\left[\langle D^iF, g\rangle _{\cH_i}\right] . \label{e.2.1}
\end{equation}

We also recall the following differentiation rule:
\begin{lemma}\label{lem.DZ}
Let $Z$ be an arbitrary random  matrix.
Assume that all   entries of $Z$ are Malliavin differentiable and $Z^{-1}$ exists almost surely.
Then
\begin{eqnarray*}
D Z^{-1}  = -Z^{-1}  ( DZ  ) Z^{-1} .
\end{eqnarray*}
If furthermore, all entries of $Z$ are twice Malliavin differentiable, then
\begin{align*}
D^{2}Z^{-1} 
=2Z^{-1}\cdot DZ \cdot Z^{-1}\cdot DZ  \cdot Z^{-1}
 -Z^{-1} \cdot D^{2}Z  \cdot Z^{-1}.
\end{align*}
\end{lemma}
\begin{proof} See p.172 in \cite{Hu2017analysis}.
\end{proof}


\subsection{Elements of Young integrals}

 We   define the H\"older continuous functions:
\begin{definition}
Let $f$ be a  continuous function   on $[0,T]$ and let  $\al \in (0,1)$. The H\"older  semi-norm is defined by
\begin{eqnarray*}
\|f\|_{\al} = \sup_{s,t\in [0,T]} \frac{|f_{t}-f_{s}|}{|t-s|^{\al}}.
\end{eqnarray*}
When $\|f\|_{\al}<\infty$,
 we call  $f$   a  H\"older continuous function of order $\al$.

\end{definition}
We have the  following result on the Young integral (see \cite{young1936inequality}):
\begin{lemma}\label{lem.young}
Let $\al, \be\in (0,1)$ be such that $\al+\be>1$. Suppose that  $f$ and $g$ are H\"older continuous of order $\al$ and    $\be$ on $[0,T]$,  respectively. Then the Young integral  $h_{t} : = \int_{0}^{t} ( f_{s}-f_{0})dg_{s}$ is well-defined for $t\in [0,T]$, and there exists a constant $C$ depending on $\al$ and $\be$ such that the following relation  holds
\begin{eqnarray*}
\|h\|_{\be} \leq  C \|f\|_{\al} \cdot \|g\|_{\be}.
\end{eqnarray*}

\end{lemma}

\section{
Dissipativity  and properties of the  solution }\label{well-posedness}
In this section, we   derive some properties of the solution for   equation \eqref{maineq} under the one-sided Lipschitz condition (see \eqref{eqn.ass1}) and the polynomial growth condition (see \eqref{eqn.ass2}).
The existence and uniqueness of the solution 
 will be proved in  Proposition \ref{prop.ode}.

\subsection{Moment bounds for equation \eqref{maineq}}
We make the following assumptions throughout the remaining  part of the paper.
In the following we denote
\begin{align}
J_b(x)\equiv \partial b
\equiv  ({\partial b_i\over \partial x_j})_{1\le i, j\le m}
\end{align}
the Jacobian of  (the vector field)  $b$, which  is a continuous
function from $\RR^m$ to the set of $m\times m$ matrices.
We also denote $\partial^{2} b(x)  = \Big( \frac{\partial^{2} b_{i}}{\partial x_{j}\partial x_{k}}(x) \Big)_{1\leq i,j,k\leq m}$ and $\partial^{3} b(x)  = \Big( \frac{\partial^{3} b_{i}}{\partial x_{j}\partial x_{k}\partial x_{\ell}}(x) \Big)_{1\leq i,j,k,\ell\leq m}$.
In the following and throughout the paper for  a vector $x \in \RR^m$, we denote $|x|$ the Euclidean norm and for  a  matrix $A$  we define its norm  by
\begin{eqnarray}\label{eqn.Anorm}
|A |=\|A\|_\infty= \sup_{|x|=1}|Ax|\,.
\end{eqnarray}

\begin{assumption} \label{ass-1}
We assume the following.

\begin{enumerate}
\item[({\bf A1})] The coefficient $b$ is one-sided Lipschitz:
namely, there is a constant $\kappa$ such that
\begin{align}\label{eqn.ass1}
\langle x-y, b(x)-b(y)\rangle \le \kappa  |x-y|^{2}, \quad
\forall \ x, y\in \RR^m.
\end{align}
\item[({\bf A2})] The coefficient $b$ itself  and its first and second  derivatives are  of polynomial  growth:
namely, there are constants $\kappa $ and $\mu $ such that
\begin{align}\label{eqn.ass2}
 |b(x)| +|\partial b(x)|+|\partial^{2} b(x)|\le \kappa  (1+|x|^{\mu }),   \quad
\forall \ x \in \RR^m.
\end{align}
\end{enumerate}
\end{assumption}

Under these assumptions  we   obtain the following results.

\begin{prop}\label{p.3.2}
Let $X_t$ be the  solution of equation  \eqref{maineq} and let   Assumption   \ref{ass-1}  be satisfied.
Then for any integer $p\ge 1$, there exists a constant  $C_{\kappa ,\mu ,p,T}$
depending  only on  $\kappa$,   $\mu $, $p$ and $T$ such that
\begin{equation}
\EE \sup_{t\in[0,T]}|X_t|^p \le C_{\kappa, \mu ,p,T}. \label{e.3.4}
\end{equation}
\end{prop}
\begin{proof}
Take $V_t=X_t-B_t $. Then   equation \eqref{maineq} becomes a deterministic equation
with random coefficient
\begin{align*}
\d V_t=b(V_t+B_t )\d t,
\end{align*}
  and can be solved in a pathwise way.
Let us differentiate   $|V_t|^2$  to obtain
\begin{eqnarray}\label{eqn.v2}
\frac{\d}{\d t} |V_t|^2
&=&2\langle V_t,\dot V_t\rangle = 2\langle V_t,b(V_t+B_t )\rangle.
\end{eqnarray}
Equation \eqref{eqn.v2} can be rewritten  as
\begin{eqnarray}\label{eqn.v21}
\frac{\d}{\d t} |V_t|^2
&=&2\langle (V_t+B_t) -B_t ,b(V_t+B_t )- b(B_t ) \rangle
+2\langle V_t,b(B_t ) \rangle.
\end{eqnarray}
Now applying  Assumption \ref{ass-1} (A1)   to the first     inner product in the right side of \eqref{eqn.v21}  and   the elementary inequality $\langle a, b  \rangle \leq \frac12 (|a|^{2}+|b|^{2}) $ and Assumption \ref{ass-1} (A2) to the second inner product, we obtain
\begin{align*}
\frac{\d}{\d t} |V_t|^2
\le&(2\kappa+1)|V_t|^2 +  \kappa ^2(1+ |B_t | ^{2\mu }  ).
\end{align*}
By Gronwall lemma it turns out that
\begin{equation*}
  |V_t|^2
\le e^{(2\kappa+1) t} |V_0|^2+\int_0^t  \kappa ^2 e^{ (2\kappa+1) (t-s)}
(1+  |B_s | ^{2\mu })\d s.
\end{equation*}
Take the square root in both sides and take into account that  $V_0=x_0$.   We obtain
\begin{equation*}
 \sup_{0\le t\le T}  |V_t|
\le e^{ (\kappa+1/2) T} |x_0|
+\sqrt{\int_0^T \kappa^2 e^{(2\kappa+1) (t-s)}
(1+  |B_s | ^{2\mu }  ) \d s}.
\end{equation*}
It then follows that
\begin{equation*}
\EE \sup_{0\le t\le T}  |V_t| ^p \le C_{\kappa, \mu, p,T},
\end{equation*}
for some finite  constant $C_{\kappa,\mu, p,T}$.
Now we use the relation
$X_t=V_t+B_t $ to obtain
\[
\EE \sup_{t\in[0,T]}|X_t |^p
\le 2^p \EE\left(\sup_{0\le t\le T} |V_t|^p+\sup_{0\le t\le T} |B_t|^p\right)
<
\infty.
\]
The proof is   completed.
\end{proof}

The next lemma considers   the H\"{o}lder continuity of  $X$.

\begin{lemma}\label{p.3.3}
Let $X_t$ be the  solution of equation  \eqref{maineq} and let the assumptions be as in Proposition \ref{p.3.2}. Take  $\al$ such that  $ 0<\al <H$.   Then
\begin{eqnarray}\label{eqn.xholder}
\|X\|_{\al} \leq  \kappa(1+\sup_{t\in [0,T]}|X_{t}|^{\mu})  +\|B\|_{\al}
\qquad
\text{and}
\qquad
\EE |X_{t}-X_{s} |^{p} \leq C |t-s|^{Hp},
\qquad p\geq 1 .
\end{eqnarray}

\end{lemma}
\begin{proof}
From the equation \eqref{maineq} satisfied by  $X_t$,  it follows
\begin{align}\label{eqn.xin}
X_t-X_s=\int_s^t b(X_u)\d u+B_t -B_s.
\end{align}
 Applying  (A2) of Assumption \ref{ass-1} and then \eqref{e.3.4} to $b(X_{u})$ and   taking H\"older norm on both sides of \eqref{eqn.xin} we obtain the   relation in~\eqref{eqn.xholder}.
 \end{proof}

\subsection{The Malliavin derivatives  for equation \eqref{maineq}}
In order to  find the   rate of convergence of our numerical schemes in the next section  we  will need   the Malliavin differentiability of equation \eqref{maineq}. To this end we make the following additional assumption on the coefficient.
\begin{assumption}\label{ass-2}
We assume the following.
\begin{enumerate}
\item[({\bf A2'})]
$\partial^{i} b(x) $, $i=0,\dots, 3$ are of polynomial  growth, namely, there are constants $\kappa, \mu  $   such that
\begin{equation}
\max_{i=0,\dots,3} |\partial^{i} b(x)|\le \kappa  (1+|x|^{\mu }),
 \quad \forall \ x \in \RR^m.
\label{eqn.ass21}
\end{equation}
\end{enumerate}
\end{assumption}

Let $r\in [0,T]$.
 Taking the Malliavin derivative $D_r$ on  both sides of equation \eqref{maineq}  leads to
\begin{align}\label{eqn.Dx}
D_{r}X_{t} = \int_{r}^{t} J_{b} (X_{s}) D_{r}X_{s} \d s + \text{Id},
\qquad \text{for}\quad t\in  [r, T],
\end{align}
and $D_{r}X_{t} = 0$ for $t<r$.  It follows from Proposition \ref{prop.ode} that there exists a unique solution for the equation \eqref{eqn.Dx} of $D_{r}X_t$.
  The following lemma is
taken from \cite{hu1996semi}.
\begin{lemma}\label{lem.jb}
The coefficient
$b$  satisfies the one-side Lipschitz
 condition  ({\bf A1})  if
\begin{align*}
\langle x, J_b(y)x\rangle \le \kappa  |x|^2,
\end{align*}
for any $x\in \RR^m$.
\end{lemma}

Lemma \ref{lem.jb} can be used to prove the following  bounds for the Malliavin derivatives  of the solution.
\begin{prop}\label{p.3.6}   Let   Assumptions  \ref{ass-1} and \ref{ass-2} be satisfied. Then
for any $p\ge 1$,  there is a constant $C_{\kappa, \mu, T, p}$ such that
\begin{equation}
\EE \left[ \sup_{0\le  r, t\le T}  |D_{r}X_{t} |^{p}\right] \le C_{\kappa, \mu, T, p}<\infty,
\label{e.3.10}
\end{equation}
and
\begin{equation}
\EE \left[ \sup_{0\le r, r',  t\le T} |D_{r  r'}^2X_{t}|^{p}\right] \le C_{\kappa, \mu, T, p}<\infty.
\label{e.3.11}
\end{equation}
\end{prop}
\begin{proof}
Fix $r$ and consider \eqref{eqn.Dx} as an equation of $(U_t:=D_rX_t, \, r\le t\le T)$.  Then
\[
\d U_t=\partial b (X_t) U_t \d t\,,\quad
t\in [r, T]\,,
\]
and $U_r=\text{Id}$.
Differentiating
$|U_t|^2$ yields
\[
\frac{\d}{\d t} |U_t|^2=2\langle U_t,  \partial b (X_t) U_t\rangle
\le 2\kappa |U_t|^2\,.
\]
Thus, we have
\begin{align*}
|D_{r}X_{t}|^{2} \leq C e^{2\kappa (t-r)}\,.
\end{align*}
Since $D_{r}X_{t}=0$ for $r<t$  we have
\begin{equation}
\sup_{0\le r, t\le T} |D_{r}X_{t}|^{2} \leq C e^{2\kappa T}\,.   \label{e.3.11a}
\end{equation}
This proves  \eqref{e.3.10}.  Now we want to prove
\eqref{e.3.11}.
 Differentiating \eqref{eqn.Dx} again, we have  for $t\geq r\vee r'$
 \begin{align}
 \label{second Malliavin Derivative}
D_{rr'}^{2}X_{t} = \int_{r\vee r'}^{t} \partial b (X_{s}) D_{rr'}^{2}X_{s}\d s
+ \int_{r\vee r'}^{t} \langle \partial^{2} b (X_{s}) , D_{r'} X_{s} \otimes  D_{r} X_{s}\rangle \d s,
\end{align}
and $D^{2}_{rr'} X_{t} = 0$ for $t< r\vee r'$,  where $(D^2_{rr'} X_t\,, t\ge  r\vee r')$ is a $m\times m$ matrix-valued process for fixed $r, r'$.
Denote the   second integral process  in
\eqref{second Malliavin Derivative}  by
\begin{align*}
g_{t}: = \int_{r\vee r'}^{t} \langle \partial^{2} b (X_{s}) , D_{r'}X_{s} \rangle D_{r} X_{s}\d s.
\end{align*}
Similar to  the argument to obtain \eqref{e.3.10}, by differentiating
$|D^{2}_{rr'} X_{t}|^{2}$ in $t$, we can show that
\begin{align*}
\sup_{0\le r, r', t\le T}|D^{2}_{rr'} X_{t}|^{2} \leq &  Ce^{2(\kappa +1) T}\,  \cdot \,
\sup_{0\le s,r, r'\le T} \left|\langle \partial^{2} b (X_{s}) , D_{r'}X_{s} \otimes  D_{r} X_{s} \rangle\right|^2.
\end{align*}
From the assumption \eqref{eqn.ass21} and our bound \eqref{e.3.10} for $D_rX_t$ it follows
\begin{align*}
\sup_{0\le r,r',t\le T}|D^{2}_{rr'} X_{t}|^{2} \leq &  C_{\kappa, T}
\sup_{0\le r,r',t\le T}  |X_t|^{2\mu}.
\end{align*}
Now combining the above with the moment bound \eqref{e.3.4} 
we conclude \eqref{e.3.11}.
\end{proof}
\begin{remark} It is interesting to point out  that we have the almost sure uniform bound \eqref{e.3.11a} for  the first Malliavin derivative of the true solution $D_rX_t$.
\end{remark}

\section{Rate of convergence for the backward Euler scheme}\label{BEM}

In this section, we apply the  backward Euler scheme to approximate the solution of  \eqref{maineq}.
The convergence rate of this scheme will also be studied. In the next section we will show that this rate of convergence is optimal.

\subsection{Preparations}
Let
\begin{eqnarray}\label{eqn.pi}
\pi: 0=t_0<t_1<\cdots<t_{n-1}<t_n=T
\end{eqnarray}
  be a partition of the time
interval $[0, T]$.   Denote
\begin{eqnarray}\label{eqn.pib}
\De_k=t_{k+1}-t_k\,,\quad \De
B_k=B_{t_{k+1}}-B_{ t_k} ,
\quad |\pi|=\sup_{0\le k\le  n-1}\De_k.
\end{eqnarray}
The classical backward Euler scheme applied  to \eqref{maineq}  is
\begin{equation}\label{eqn.back.euler}
Y_{t_{k+1}}^\pi =Y_{t_k} ^\pi +b(Y_{t_{k+1}}^\pi) \De_k +\De B_k, \quad k=0, 1, \cdots, n-1\,.
\end{equation}
This is  an implicit scheme.  To find
$Y_{t_{k+1}}^\pi$ from $Y_{t_k}^\pi$,  we need to solve a function  equation:
$Y_{t_{k+1}}^\pi -b(Y_{t_{k+1}}^\pi) \De_k
 =Y_{t_k} ^\pi + \De B_k$.
This scheme gives all the values of the approximation at the partition points
 $t_0, t_1, \cdots, t_n$.
To compare the approximation solution  with the original solution,
we need to
know the value of the approximation solution at all time instant. We shall use the following interpolation:
let $Y_t^\pi$ satisfy
\begin{align}
Y_{t}^\pi=Y_{t_k} ^\pi +b(Y_{t}^\pi) (t-{t_k})+ (B_t - B_{t_k}),
\quad t\in (t_k, t_{k+1}],  \quad k=0, 1, \cdots, n-1. \label{e.4.2}
\end{align}
To  simplify the notation in the remaining part of the paper,   we  omit the explicit dependence of
$Y_t\equiv Y_t^\pi$ on the partition $\pi$.

Before we proceed to study the rate of convergence, let us recall   the following preparatory results; see \cite[Lemma 2.2]{hu1996semi}.

\begin{lemma}  \label{lem.jinvb}
If $J$ is an  $m\times m$ square matrix such that
\begin{align}\label{eqn.lam}
\langle x, Jx\rangle\le \la |x|^2,\quad \forall x\in \RR^m,
\end{align}
then  for all $t\in \RR$ such that  $\la t<1$, $I-tJ$ is invertible and
\begin{align}
\| (I-tJ)^{-1} \|_\infty \le (1-\la t)^{-1}, \label{e.4.4}
\end{align}
where  we recall that $\|\cdot \|_\infty$   means the operator norm for an operator
from $\RR^m$ to $\RR^m$ (see \eqref{eqn.Anorm}).
\end{lemma}

 We also have the following result.
\begin{lemma}\label{corollary.I-tJ}
  Assume that  the coefficient $b$ satisfies the one-sided Lipschitz
condition ({\bf A1}) for some  $\kappa>0$, and    that $b$  is continuously differentiable.     Then when
$t\kappa <1$,   the matrix
$I-t\int_{0}^{1}J_b(x+y\theta) \d \theta  $
 is invertible  for any $x,y\in \RR^m$   and we have

\begin{align}\label{eqn.jb}
\left
\|\left(I-t\int_{0}^{1}J_b(x+y\theta) \d \theta \right)^{-1} \right\|_\infty \le (1-\kappa  t)^{-1}.
\end{align}
\end{lemma}
\begin{proof}
Set  $J= \int_{0}^{1}J_b(x+y\theta) \d \theta$. It follows from Lemma 2.1 in \cite{hu1996semi} that    $J$ satisfies   condition \eqref{eqn.lam} with $\lambda=\kappa$. Relation  \eqref{eqn.jb} is then obtained  by applying  Lemma \ref{lem.jinvb} to $J$.
\end{proof}

\subsection{Properties of the approximate solutions}

We  will need   some properties of the  approximated  solution generated by the  backward Euler scheme
\eqref{eqn.back.euler}   for our convergence result.
In this subsection we are concerned  with the upper-bound estimates of the approximate solutions.

\begin{prop}\label{lem.numbound}
 Let   Assumptions  \ref{ass-1} and \ref{ass-2} be satisfied.  Let $Y_t$ be the solution of  the backward Euler scheme defined  by   \eqref{e.4.2}. Then  for any $p\ge 1$, we have
\begin{align}\label{eqn.ynb}
\EE \sup_{0\le  t\le T}|Y_{t }|^p<\infty.
\end{align}
\end{prop}
\begin{proof}
We shall show $\EE \sup_{0\le  k\le n}|Y_{t_k }|^p<\infty$  and \eqref{eqn.ynb} can be proved similarly. The   proof  can be considered as a discretized  version of    that of   \eqref{e.3.4}.

Denote $\widetilde{Y}_{t_k}=Y_{t_k}-B_{t_k}$ for $0\le  k \le n $. Then we  can    write  equation   \eqref{eqn.back.euler}
as
\begin{align}\label{eqn.yt}
\widetilde{Y}_{t_{k+1}}=&\widetilde{Y}_{t_k}+\left[ b(\widetilde{Y}_{t_{k+1}}+B_{t_{k+1}})-b(B_{t_{k+1}})\right]\De_k +b(B_{t_{k+1}})\De_k
\nonumber
\\
=&\widetilde{Y}_{t_k}+ \De_k \cdot
\int_{0}^{1}
J_b(\theta\widetilde{Y}_{t_{k+1}}+ B_{t_{k+1}})
\d \theta \cdot
\widetilde{Y}_{t_{k+1}} +b(B_{t_{k+1}})\De_k.
\end{align}
Solving equation \eqref{eqn.yt} for $\widetilde{Y}_{t_{k+1}} $ we obtain
 \begin{align}\label{eqn.yk1}
 \widetilde{Y}_{t_{k+1}} = \beta_k(\widetilde{Y}_{t_k}+b(B_{t_{k+1}})\De_k)\,,
 \end{align}
 where
 $$\beta_k=\lc I-\De_k\cdot \int_{0}^{1}
J_b(\theta\widetilde{Y}_{t_{k+1}}+ B_{t_{k+1}})
\d\theta \rc^{-1}.
$$
Iterating \eqref{eqn.yk1} we obtain
 \begin{align}\label{eqn.yk2}
 \widetilde{Y}_{t_{k+1}} &=
 \sum_{j=0}^{k}
 \beta_k\cdots \beta_{j} b(B_{t_{j+1}} )\De_j +   \beta_k\cdots \beta_{0} Y_{0}.
 \end{align}
Let us apply  Lemma  \ref{corollary.I-tJ} to $\be_{k}$ with $t=\De_k$,
$x=B_{t_{k+1}}$ and $y=\widetilde{Y}_{t_{k+1}}$. This  yields
\begin{eqnarray}\label{eqn.be}
|\beta_k|_\infty \le(1-\kappa \De_k)^{-1} \leq e^{2\kappa \Delta_{k}}.
\end{eqnarray}
 Applying the estimate   \eqref{eqn.be} and   condition (A2) for $b$ to \eqref{eqn.yk2} we obtain
\begin{align}\label{eqn.yk3}
|\widetilde{Y}_{t_{k+1}}|  \le& e^{2\kappa T}  \sum_{j=0}^{k} \sup_{0< j\leq k-1} | b(B_{t_{j}}) | \Delta_{j} +  e^{2\kappa T}| {Y}_{0}|
\nonumber
\\
\le&  e^{2\kappa T}\ T\left(1+\sup_{0< j \le k-1}|B_{t_j}|^{\mu }\right) +  e^{2\kappa T}| {Y}_{0}|.
\end{align}
The estimate \eqref{eqn.ynb} follows by taking sup for $k=0,\dots,n$   and then taking expectation in both sides of \eqref{eqn.yk3}.
\end{proof}

\subsection{Malliavin derivatives of the numerical schemes}
In this subsection we consider
the Malliavin derivatives for  backward Euler scheme under the
one-sided Lipschitz condition.

\begin{prop}\label{lem.dy} Let $b$ satisfy   Assumption
\ref{ass-1} and \ref{ass-2} and let $Y_{t }$ be
defined by \eqref{e.4.2}.   Then   for any $p\ge 1$, there is a constant $C_{p, T}<\infty$  independent of the partition $\pi$ such that
\begin{eqnarray}
\EE \sup_{s,t\in [0,T]}
|D_{s}Y_{t}|^p  \leq C_{p, T}
\qquad\text{and}\qquad
\EE \sup_{r,s,t \in [0,T] }|D_{rs}^2Y_{t} |^p\leq C_{p, T} . \label{e.4.9}
\end{eqnarray}
\end{prop}

\begin{proof}

Take $r\in [0,T]$.
Note that since $Y $ is  adapted to the filtration generated by $B $ we have $D_rY_{t }=0$ for all $r>t$.
In the following we focus on the case when    $r\le t$. Let $k_0$ be  such that $t_{k_0}< r\leq  t_{k_0+1}$.
Take the Malliavin derivative $D=(D^1, \cdots, D^m)$ on both sides of \eqref{e.4.2}. This  yields
\begin{align}\label{eqn.dyk}
D_rY_{t}&=D_{r}Y_{\eta(t)}
+J_b(Y_{t  })D_rY_{t  } (t-{\eta(t)} ) +\text{Id}\cdot \mathbf{1}_{(\eta(t),t]}(r)
 ,
\end{align}
where we denote \begin{eqnarray}\label{eqn.eta}
 \eta(t) = t_{k_{1}}\qquad\text{and}\qquad    k_{1}=\max \{ {i}:   t_i< t \}.
\end{eqnarray}
Solve equations in \eqref{eqn.dyk} for $D_{r}Y_{t}$ we get
\begin{eqnarray}
D_{r}Y_{t} = (  \text{Id} - J_{b}(Y_{t}) (t-\eta(t)) )^{-1}  \cdot ( D_{r}Y_{\eta(t)} + \text{Id}\cdot \mathbf{1}_{(\eta(t),t]}(r) ).
\label{eqn.dy2}
\end{eqnarray}
Iterating \eqref{eqn.dy2}   we  obtain
\begin{align}
 D_rY_t
=
\alpha_{t} \alpha_{\eta(t)}\cdots   \alpha_{t_{k_{0}+1}}   ,
 \label{e.4.10}
\end{align}
where
\begin{align*}
\alpha_{t}=(1-(t-
\eta(t)) J_b(Y_t))^{-1}.
\end{align*}
Now applying  Lemma  \ref{corollary.I-tJ} to $\al_{t}$ with  $y=0$ and $x=Y_{t}$ we get
\begin{align}
 \|D_rY_t \|
\le & \| \alpha_{t }\|  \cdot  \| \alpha_{
\eta(t) }\|
\cdots \|\alpha_{t_{k_{0}+1}}\|   \nonumber \\
\le&(1-\Delta_{k_{1}} \kappa )^{-1}(1-\Delta_{k_{1}-1} \kappa )^{-1}
\cdots(1-\Delta_{k_{0}} \kappa )^{-1}\nonumber \\
\le&e^{2\kappa (\Delta_{k_{1}}+\Delta_{k_{1}-1} +\cdots+\Delta_{k_{0}})} \le e^{2\kappa (t-r)}. \label{e.4.11}
\end{align}
where $\|\cdot\|$ is defined in \eqref{eqn.Anorm}.
This proves  the first bound in
 \eqref{e.4.9}.

We turn to  bound  the  second Malliavin derivative of $Y_t$, namely, to prove the second bound in~\eqref{e.4.9}.
We shall use \eqref{e.4.10} to compute
$D_{rr'}^2Y_t$. For convenience   we consider $D^j_{r'} (D_rY_t)$ for $j=1, \cdots, m$.
By the product rule of the Malliavin derivative operator, we have
\begin{align*}
  D_{r'} ^j (D_rY_t)
 =&D_{r'} ^j \{ \alpha_{t} \alpha_{\eta(t)}\cdots   \alpha_{t_{k_{0}+1}}  \} \\
=& \sum_{k=k_{0}+1}^{k_{1}} \alpha_{t} \alpha_{\eta(t)}\cdots \alpha_{t_{k+1}}
  (D_{r'}^{j}\alpha_{t_k}) \alpha_{t_{k-1}}\cdots\alpha_{t_{k_{0}+1}}\\
  &\qquad
+ (D_{r'} ^j   \alpha_{t }) \alpha_{\eta(t)}
\cdots \alpha_{t_{k_{0}+1}} .
\end{align*}
To compute $D_{r'}^j\al_{t_{k}}$ or
$D_{r'}^j\al_t$  we use Lemma \ref{lem.DZ}.
\begin{align*}
D_{r'}^j\al_{t_{k}}
=
&-(1-\Delta_{k-1} J_b(Y_{t_k}))^{-1}  \left(
D_{r'}^j  (1-\Delta_{k-1} J_b(Y_{t_k})) \right)
(1-\Delta_{k-1} J_b(Y_{t_k}))^{-1} \\
=&-(1-\Delta_{k-1} J_b(Y_{t_k}))^{-1}  \left(-\De_{k-1} \partial^2 b (Y_{t_k}) D_{r'}^j Y_{t_k}\right)
(1-\Delta_{k-1} J_b(Y_{t_k}))^{-1}.
\end{align*}
 Hence,
 \begin{align}
  \|D_{r'} ^j (D_rY_t) \|
\le & \sum_{k=k_0+1}^{k_{1}} \|\alpha_{t } \|
\|\alpha_{\eta(t)}\|  \cdots \|\alpha_{t_{k+1}}\|   \left\| D_{r'}^j\alpha_{t_k} \right\|  \|\alpha_{t_{k-1}}\| \cdots
  \|\alpha_{k_0+1}\|
  \nonumber
   \\
  &\qquad \quad
+ \left\| D_{r'} ^j   \alpha_{t }  \right\|  \|\alpha_{
\eta(t) }\|
\cdots \|\alpha_{k_0+1} \|
\nonumber
\\
\le& e^{2\kappa (t-r) }
 \sum_{k=k_0+1}^{k_{1}}
  \left\| D_{r'}^j\alpha_{t_k} \right\|
+  e^{2\kappa (t-r) } \left\| D_{r'} ^j   \alpha_{t }  \right\|     ,
\label{eqn.ddy}
\end{align}
where in the   last inequality we used   bounds analogous to \eqref{e.4.11}.
By   Assumption \ref{ass-1} and \ref{ass-2} and then \eqref{e.4.11}  we have
\begin{align}
\|D_{r'}^j\al_{t_k}\|
\le &   C  \De_{k-1}  (1+ |Y_{t_k}|^{3\mu})
||D_{r'}^j Y_{t_k}
||
\nonumber
\\
\le &   C  \De_{k-1} (1+ |Y_{t_k}|^{3\mu}).
\label{eqn.dal}
\end{align}
The same  bound holds for $\|D_{r'}^j\al_t\|  $. The proof is   similar to $\|D_{r'}^j\al_{t_{k}}\| $ and is   omitted.

Substituting  \eqref{eqn.dal} into
\eqref{eqn.ddy} we obtain
\begin{align*}
  \|D_{r'} ^j (D_rY_t) \|
\le& C e^{2\kappa (t-r) }
 \sum_{k=k_0+1}^{k_{1}}     \De_{k-1}  (1+ |Y_{t_k}|^{3\mu})   +  Ce^{2\kappa (t-r) }    \De_{k_{1}} (1+ |Y_{t }|^{3\mu})    \\
\le& C  (1+ |Y_{t }|^{3\mu}).
\end{align*}
Finally, applying \eqref{eqn.ynb}   we have
for any $p\ge 1$
\begin{align*}
\EE \sup_{r,r',t}\|D_{r'} ^j (D_rY_t) \|  ^p
\le &   C   \EE\sup_{t}\left(1+ |Y_{t}|^{3p\mu}\right)   <\infty.
\end{align*}
This proves the second inequality in \eqref{e.4.9}.
\end{proof}

\subsection{Estimate   for $\EE \lc
F\Delta B_{uv} ^i
\rc$  and $\EE \lc
F\Delta B_{uv} ^i\Delta B_{st}^j
\rc$}

In this subsection we derive an  upper-bound estimate result for the increments of the fBm $B_t$, which will be useful  in the proof of rate of convergence in Section \ref{subsection.rate}. For convenience let us denote
\begin{equation}
\Delta B_{uv}^i=
B_v^i-B_u^i
\end{equation}
for $0\le u\le v\le T$ and $i=1,\dots, m$.

\begin{lemma} \label{lemma.FDe_B}
Let $F:\Om\to \RR$ be a random variable which possesses  second Malliavin derivative.
If $\EE(|F|)<\infty$ and if $\sup_{s\in[0,T]} \EE |D_s^iF|<\infty$,  then
\begin{eqnarray}
\left|\EE \lc
F\Delta B_{uv} ^i
\rc\right| \leq C\cdot  |u-v|.
\label{e.2.3}
\end{eqnarray}
 Assume further that  $\sup_{s, t\in[0,T]} \EE |D_{s t}^{ij}F|<\infty$. Then  for any $0\le u<v\le T$ and   $0\le s<t\le T$
\begin{eqnarray}
\left| \EE \lc
F\Delta B_{uv} ^i\Delta B_{st}^j
\rc \right|
\leq  C \cdot\langle \mathbf{1}_{[s,t]},  \mathbf{1}_{[u,v]} \rangle + C \cdot |v-u| \cdot |t-s|
.\label{e.2.4}
\end{eqnarray}
\end{lemma}

\begin{proof}
Applying \eqref{e.2.1} with $g(t)=\mathbf{1}_{(u, v]}(t)$ yields
\begin{align*}
\left|\EE \lc F\Delta B^{i}_{uv}\rc
\right|
=&\left|
\EE \int_0^T \int_0^T D_s^iF \mathbf{1}_{(u, v]}(t)
\phi_i(s,t)  \d s\d t\right| \\
\le & \sup_{0\le s\le T} \EE |D_s^iF|  \cdot
\int_u^v \int_0^T \phi_i(s,t) \d t \d s .
\end{align*}
Now we have
\begin{equation}
\int_0^T \phi_i(s,t) \d t=\frac{H_i}{2H_i-1}
\left[ (T-s)^{2H_i-1}+s^{2H_i-1}\right]\le CT^{2H_i-1}. \label{e.2.5}
\end{equation}
Thus    \eqref{e.2.3} is proved.

Next, we turn to prove \eqref{e.2.4}.
Applying \eqref{e.2.1} twice we obtain
\begin{align}\label{eqn.fbb}
\left|\EE \lc
F\Delta B_{uv} ^i\Delta B_{st}^j
\rc \right|
& =\left|\EE \left[ \langle D^{ij} F,
\mathbf{1}_{(u,v]\times (s, t]}\rangle _{\cH_i\otimes \cH_{j} }\right] +\EE \left[
F\langle \mathbf{1}_{(u, v]}, \mathbf{1}_{(s, t]}\rangle _{\cH_i}\right] \cdot \mathbf{1}_{\{i=j\}}\right|
\nonumber
\\
&= \Bigg|  \EE \left[  \int_0^T\int_0^T
\int_u^v\int_s^t D^{ij}  _{x \xi}F
\phi_i(x,y)\phi_j(\xi,\zeta) \d x\d y\d\xi \d\zeta\right]
\nonumber
\\
&\qquad\quad +
\EE \left[
F\right]\langle \mathbf{1}_{(u, v]}, \mathbf{1}_{(s, t]}\rangle _{\cH_i}\cdot \mathbf{1}_{\{i=j\}} \Bigg|\\
&\le \EE \sup_{0\le x,\xi\le T} |D^{ij}_{x \xi}
F |\int_0^T\int_0^T
\int_u^v\int_s^t
\phi_i(x,y)\phi_j(\xi,\zeta) \d x\d y\d\xi \d\zeta
\nonumber
\\
&\qquad\quad +
\EE \left|
F\right|\langle \mathbf{1}_{(u, v]}, \mathbf{1}_{(s, t]}\rangle _{\cH_i}   .
\end{align}
As in  \eqref{e.2.5} we have
\begin{eqnarray}\label{eqn.phi2}
\int_0^T\int_0^T
\int_u^v\int_s^t
\phi_i(x,y)\phi_i(\xi,\zeta) \d x\d y\d\xi \d\zeta\le |v-u| \cdot |t-s|
.
\end{eqnarray}
Substituting \eqref{eqn.phi2} into \eqref{eqn.fbb} we obtain the   estimate \eqref{e.2.4}.
\end{proof}

\subsection{Rate of convergence}\label{subsection.rate}
After these preparations, we return to explore the  convergence of the backward Euler scheme, namely  the convergence of $Y_t$ defined by \eqref{e.4.2} to the solution $X_t$ defined by the stochastic differential
equation \eqref{maineq}.  We shall also obtain the rate of this convergence,
which will be shown to be optimal in the next section.

In order to complete the proof, we introduce the following  notations and essential lemmas. Recall that $ t_{k}$, $k=0,\dots, n $ and $\pi$ are defined in \eqref{eqn.pi}-\eqref{eqn.pib} and   $\eta(t)$ is defined in \eqref{eqn.eta}.
We define processes $\Ga$ and $A$ on $[0,T]$:
\begin{align}
\Ga_t= & I-(t-\eta(t))\int_0^1 J_b(uX_t+(1-u)Y_t)\d u, \label{eqn.Gammat} \\
A_t=&\Ga_t^{-1}=\left( I-(t-\eta(t))\int_0^1 J_b(uX_t+(1-u)Y_t)\d u\right) ^{-1}.
\label{eqn.At}
\end{align}
 For convenience we will  also denote  $A_{k+1}=A_{t_{k+1}}$ for $k=0,\dots, n$. Resorting to Lemma \ref{lem.jb} it is easily   seen the existence of $A_t$ when $\kappa  |\pi| <1$. Furthermore, using Lemma \ref{corollary.I-tJ}   has  the following lemma for~$A_t$.
\begin{lemma}\label{At}
Let $A_t$ be defined by  \eqref{eqn.At}.  Then
\begin{align}
\|A_t\|  \le (1-\kappa (t-\eta(t)))^{-1} .  \label{eqn.bound At}
\end{align}
If the partition $\pi$ satisfies  $\kappa  \cdot |\pi| <1 $, then it has
\begin{align}\label{eqn.bound At2}
\|A_t\|   \leq  e^{ 2 \kappa  (t-\eta(t))}.
\end{align}
\end{lemma}

In the following we consider   the Malliavin derivatives of $\Ga$.
\begin{lemma}\label{Malliavin Gamma}
Let $\Ga $ be defined  in
 \eqref{eqn.Gammat}.  Then there exists a constant $C$ independent of $|\pi|$ such that
\begin{align}
\EE\left[
\sup_{ \xi,t\in [0,T]}|D_\xi\Gamma_{t}|^{p}
\right] \le C   |\pi|^{p} \quad \text{and}
\quad
\EE\left[
\sup_{ \xi , \zeta,t\in [0,T] }|D_{\xi\zeta}^2\Gamma_{t}|^{p}\right] \le C |\pi|^{p}. \label{e.4.15}
\end{align}
\end{lemma}
\begin{proof}
By  the definition of $\Ga_t$    we have
\begin{align*}
 D_\xi^j \Gamma_{t} =&-(t-\eta(t))\int_{0}^{1} D^{j}_\xi  J_b(uX_t+(1-u)Y_t) \d u\\
=&-(t-\eta(t))\int_{0}^{1} \partial^{2}b(uX_t+(1-u)Y_t)\cdot(uD_\xi^j X_t+(1-u)D_\xi^j Y_t) \d u \, .
\end{align*}
Applying Assumption \ref{ass-1} (A2) and then using the elementary inequality $ab\leq \frac12 (a^{2}+b^{2})$  and taking into account that $t-\eta(t)\leq |\pi|$ we obtain
\begin{align*}
\| D_\xi^j  \Gamma_{t}\|
\le& C (t-\eta(t))
(   1+| X_t|^{2 \mu } + |Y_t|^{ 2\mu } +   |D_\xi^j X_t|^{2}  + |D_\xi^j Y_t|^{2}  ).
\end{align*}
Apply  Propositions \ref{p.3.2},
\ref{p.3.6}, \ref{lem.numbound},  \ref{lem.dy} to $X_{t}$, $D^{j}_{\xi}X_{t}$, $Y_{t}$ and $D_\xi^j Y_t$, respectively. We  obtain
\begin{align*}
\EE\left[
\sup_{\xi,t\in [0,T]}|D_\xi\Gamma_{t}|^{p}\right]
\leq C_p |\pi|^p, \qquad \text{for all } \ p\ge 1.
\end{align*}

This is the first inequality in \eqref{e.4.15}.

Differentiating \eqref{eqn.Gammat} twice gives
\begin{align*}
 D_{\xi\zeta}^{ij}\Gamma_{t}  =&-(t-\eta(t))\int_{0}^{1} \partial^{3}b(uX_t+(1-u)Y_t) \\
 &\quad \cdot   (uD_\xi^i X_t+(1-u)D_\xi^i Y_t) \otimes (uD_\zeta^j X_t+(1-u)D_\zeta^j Y_t) \d u\\
&-(t-\eta(t))\int_{0}^{1} \partial^{2}b(uX_t+(1-u)Y_t)\cdot(uD_{\xi\zeta}^{ij}  X_t+(1-u)D_{\xi\zeta}^{ij}  Y_t) \d u.
\end{align*}
Thus applying   Assumption \ref{ass-2}  on the coefficient $b$ we   have
\begin{align*}
 \|D_{\xi\zeta}^2\Gamma_{t}\|
 \le &C |\pi|\left(   1+|X_t|^\mu +|Y_t|^\mu \right) \\
 &\qquad   \bigg(    (|D_\xi  X_t|  + |D_\xi  Y_t|)
  (|D_\xi  X_t|  + |D_\xi  Y_t|)
   +
 |D_{\xi\zeta}^2 X_t| +|D_{\xi\zeta}^2 Y_t| \bigg)
  \\
  \le &C |\pi|\bigg(   1+|X_t|^{2\mu} +|Y_t|^{2\mu}   +    |D_\xi X_t|^4 + |D_\xi Y_t|^4\\
 &\qquad   +   |D_\zeta X_t|^4 + |D_\zeta Y_t|^4 +
 |D_{\xi\zeta}^2 X_t|^2+ |D_{\xi\zeta}^2 Y_t|^2\bigg).
\end{align*}
  Applying Propositions \ref{p.3.2} ,
\ref{p.3.6}, \ref{lem.numbound},  \ref{lem.dy}  again yields the second inequality of \eqref{e.4.15}. This proves the lemma.
\end{proof}

Let us now turn to the estimate of the derivatives of $A$.
\begin{lemma}\label{Malliavin A}

Let $A_t$ be defined  by  \eqref{eqn.At}.  Then   there exists a constant $C$ independent of $|\pi|$ such that
\begin{align}
\EE\left[
\sup_{\xi,t \in [0,T] }|D_\xi A_{t}|^{p}
\right] \le C   |\pi|^{p} \quad\quad \text{and}
\quad\quad
\EE\left[
\sup_{\xi,\zeta,t\in[0,T]}|D_{\xi\zeta}^2 A_{t}|^{p}\right] \le C  |\pi|^{p}. \label{e.4.16}
\end{align}
\end{lemma}
\begin{proof} Differentiating $A$ and taking into account the relation  $A_{t} = \Ga_{t}^{-1}$ and  Lemma \ref{lem.DZ}, we have
\begin{align*}
D_{\xi}{A_t} &=-A_t \cdot D_{\xi}
\Gamma_t \cdot A_t.
\end{align*}
Applying Lemma \ref{At}-\ref{Malliavin Gamma} for $A$ and $D\Ga$ we thus obtain the first inequality in \eqref{e.4.16}.

Differentiating $A$ again we obtain
\begin{align*}
 D_{\xi \zeta}^{2}{A_t} =&  {A_t}\cdot D_{\xi}\Gamma_t\cdot{A_t}\cdot D_{\zeta}\Gamma_t \cdot{A_t} +
 {A_t}\cdot D_{\zeta}\Gamma_t\cdot{A_t}\cdot D_{\xi}\Gamma_t \cdot{A_t} - {A_t}\cdot D_{\xi \zeta}^{2}\Gamma_t \cdot{A_t}.
\end{align*}
Then  applying Lemma \ref{At}-\ref{Malliavin Gamma}  again we obtain  the
 second inequality    in \eqref{e.4.16}.
\end{proof}

For $s,t\in [0,T]$ and $j=0,\dots, n$ we define
\begin{align}
b_{1}(s, j)=\int_{0}^{1}J_{b}(u X_{t_{j+1}\wedge t}+(1-u)X_{s}) \d u\,. \label{e.4.17}
\end{align}

We will need the following estimate for $b_{1}$. The proof follows  from  a similar procedure as for $\Ga$ in Lemma~\ref{Malliavin Gamma} and is omitted.
\begin{lemma} \label{Malliavin b1}
Let $b_1$ be defined by \eqref{e.4.17}. Then for any $p\ge 1$, we have
\begin{align*}
&\EE
\left[  \sup\limits_{ s \in [0,T] \atop  j=0,\dots,n}|   b_{1}(s,j)|^{p} \right]  <\infty,
\qquad\qquad
\EE
\left[  \sup\limits_{ s,\xi \in [0,T] \atop  j=0,\dots,n}|D_{\xi} b_{1}(s,j)|^{p}  \right]  <\infty,
\\
&\EE
\left[   \sup\limits_{ s,\xi, \zeta \in [0,T] \atop  j=0,\dots,n}|  D^{2}_{\xi \zeta} b_{1}(s,j)|^{p} \right]  <\infty
.
\end{align*}
\end{lemma}

Now we can state one of the main results in this paper.

\begin{theorem}
Assume that Assumption \ref{ass-1} and \ref{ass-2} hold. Let $X_t$ satisfy
\eqref{maineq} and let $Y_t$ satisfy~\eqref{e.4.2}.
There is a constant  $C$  independent of  the partition $\pi$ (but
dependent on $ \kappa $, $\mu$,  $T$) such that
\begin{align}\label{eqn.errorT}
\EE\left[ \sup_{0\le t\le T} |Y_t-X_t|^2\right] \le C |\pi|^2,
\end{align}
and
\begin{align}\label{eqn.errort}
\int_0^T \EE |Y_t-X_t|^2\text{d} t \le C |\pi|^2.
\end{align}
\end{theorem}

\begin{proof}  We will first prove
 \eqref{eqn.errorT}. The second statement  \eqref{eqn.errort} follows from \eqref{eqn.errorT} easily. The proof is divided into several steps.

 \medskip
 \noindent\emph{Step 1: A representation for the error process.} We can write  \eqref{maineq} as
\begin{align}
X_{t}=&X_{t_k}+\int_{t_k}^t b(X_s)\textrm{d}s
       +(B_t - B_{t_k})\nonumber\\
    =&   X_{t_k}+b(X_t)(t-t_k)+(B_t -B_{t_{k}})+ R_k(t)\,,  \label{eqn.sol2}
\end{align}
where $t_k\le   t\le t_{k+1}$, and
\begin{eqnarray}\label{eqn.rkt}
R_k(t)=\int_{t_k}^t b(X_s)-b(X_t)\d s.
\end{eqnarray}
For ease of notation, we denote
$R_k:= R_k(t_{k+1})$ and $R_{k}^{t} := R_{k} (t_{k+1}\wedge t) $.

Denote the error process $Z_t=X_t-Y_t$. For convenience   we will also  denote
$Z_k=Z_{t_k}$.
Substracting
  \eqref{e.4.2} from \eqref{eqn.sol2},  we get for $t_k\le  t\le  t_{k+1}$
$$
Z_t=Z_k+(t-t_k)\{ b(X_t)-b(Y_t)\}+ R_k(t).
$$
We can rewrite the above formula as
$$
Z_t-(t-t_k)\{ b(X_t)-b(Y_t)\}=Z_k+ R_k(t).
$$
That is
\begin{eqnarray}
\label{eqn.zk}
Z_t-(t-t_k)\lcl \int_0^1 J_b(uX_t+(1-u)Y_t)\d u \rcl Z_t=Z_k+R_k(t).
\end{eqnarray}

Recall that $A$ is defined in \eqref{eqn.At} and we denote $A_{k}=A_{t_{k}}$. Therefore, multiplying $A_{t}$ on both sides of \eqref{eqn.zk} we get
$$
Z_t=A_t Z_k+A_t R_k(t).
$$
Taking $t=t_{k+1}$ particularly  arrives at
\begin{eqnarray}
\label{eqn.ztk}
Z_{k+1}=A_{k+1}Z_k+A_{k+1}R_k.
\end{eqnarray}
Denote
\begin{eqnarray}\label{eqn.atj}
A_j^t=A_{t} \lc \prod_{\ell=j+1}^{ {k}} A_{\ell} \rc,
\end{eqnarray}
where $ {k}$ is such that  $ t_k<t\le t_{k+1} $ and here  we use the convention that  if $\{\ell:j+1\le \ell \le  {k} \}=\emptyset$ then $\prod_{\ell=j+1}^{ {k}} A_{\ell}=1$. By
iterating \eqref{eqn.atj} and then using \eqref{eqn.ztk} and $R_{k}^{t} := R_{k} (t_{k+1}\wedge t) $   we obtain
\begin{eqnarray}\label{eqn.zt}
Z_{t} = \sum_{j=0}^{k} A_{j}^{t} R_{j}^{t} .
\end{eqnarray}

\smallskip
 \noindent\emph{Step 2: Estimate of $A^{t}_j$.}
 Due to \eqref{eqn.bound At},  we have
\begin{align}\label{eqn.Ajt}
\|A_j^t\|
&\le  (1-\kappa  (t-t_{ {k}}))^{-1} \cdots (1-  \kappa \De_{j+1})^{-1} (1-\kappa  \De_{j})^{-1}\nonumber\\
&\le e^{2\kappa (t-t_j)}\le C.
\end{align}
This can be used  to obtain some rate of convergence
as in \cite{hu1996semi}. However, this rate will not be optimal. To obtain the optimal rate estimate, we
need to bound the Malliavin derivative of $A_j^t$ as
in \cite{hu2016rate}.
A straightforward computation for  $D_{\xi} A_j^t$ and $D_{\xi\zeta}^2 A_j^t$  yields
\begin{align*}
D_{\xi} A_j^t=&D_{\xi}A_{t} \lc \prod_{\ell=j+1}^{ {k}}  A_{\ell} \rc + \sum_{\ell=j+1}^{ {k}}A_{t} \lc A_{j+1}\cdots A_{\ell-1}\cdot D_{\xi}A_{\ell}\cdot A_{\ell+1}\cdots A_{ {k}} \rc;
\\
D_{\xi\zeta}^2A_j^t=&D_{\xi \zeta}^2A_{t} \lc \prod_{\ell=j+1}^{ {k}}  A_{\ell} \rc  +\sum_{\ell=j+1}^{ {k}}A_{t} \lc A_{j+1}\cdots A_{\ell-1}\cdot D_{\xi \zeta}^2A_{\ell}\cdot A_{\ell+1}\cdots A_{ {k}}\rc \\
& \qquad +\sum_{\ell=j+1}^{ {k}} D_{\xi}A_{t} \lc A_{j+1}\cdots A_{\ell-1}\cdot D_{\zeta}A_{\ell}\cdot A_{\ell+1}\cdots A_{ {k}}\rc \\
& \qquad +\sum_{\ell=j+1}^{ {k}} D_{\zeta}A_{t} \lc A_{j+1}\cdots A_{\ell-1}\cdot D_{\xi}A_{\ell}\cdot A_{\ell+1}\cdots A_{ {k}}\rc \\
&\qquad +\sum_{\ell\neq\ell', \ell,\ell'=j+1}^{ {k}}A_{t} \lc A_{j+1}\cdots  D_{\xi}A_{\ell}\cdots D_{\zeta}A_{\ell'}\cdots A_{ {k}} \rc.
\end{align*}
These two formulas combined with \eqref{eqn.bound At2}, \eqref{e.4.16} and
\eqref{eqn.Ajt} yield
\begin{eqnarray}
\EE \left[ \sup_{0\le \xi, t\le T,
j=0, 1, \dots, n}|D_{\xi} A_j^t|^p\right] \le C_p<\infty; \label{e.4.25}\\
\EE \left[ \sup_{0\le \xi, \zeta, t\le T,
j=0, 1, \dots, n}|D_{\xi \zeta}^2 A_j^t|^p\right] \le C_p<\infty  \label{e.4.26}
\end{eqnarray}
for any $p\ge 1$.

\medskip
 \noindent\emph{Step 3: A decomposition for the error process.}
Now we come back to the main proof.
We write  \eqref{eqn.zt} as
\begin{align*}
Z_{t}=\sum_{j=0}^{ {k}} A_{j}^{t} R_{j}^{t}\,.
\end{align*}
Thus, we have
\begin{align}
\EE(|Z_{t}|^2)= &\EE \lc \Big|\sum_{j=0}^{ {k}} A_{j}^{t} R_{j}^{t}\Big|^2 \rc\nonumber\\
=&\sum_{j,j'=0}^{ {k}}\EE\left\{A_{j}^{t}R_{j}^{t}\cdot A_{j'}^{t} R_{j'}^{t}\right\}
=:\sum_{j,j'=0}^{ {k}} I_{j, j'}.
\end{align}
It is worthy to  notice from \eqref{eqn.rkt} that
\begin{align*}
R_j^t=&\int_{t_{j}}^{t_{j+1}\wedge t} \lcl\int_{0}^{1}J_{b}(u X_{t_{j+1}\wedge t}+(1-u)X_{s})\d u \rcl \Delta X_{s,t_{j+1}\wedge t}\d s\\
=&\int_{t_{j}}^{t_{j+1}\wedge t}b_{1}(s, j)\Delta X_{s,t_{j+1}\wedge t}\d s,
\end{align*}
where $\Delta X_{s,r}=X_r-X_s$ and recall that $b_{1}$ is defined by  \eqref{e.4.17}.
So we can write
\begin{align}
 I_{j, j'} =&\int_{t_{j}}^{t_{j+1}\wedge t}\int_{t_{j'}}^{t_{j'+1}\wedge t}
 \EE\bigg( A_{j}^{t}b_{1}(s, j)
 \Delta X_{s,t_{j+1}\wedge t} \cdot A_{j'}^{t} b_{1}(s', j')\Delta X_{s', t_{j'+1}\wedge t}\bigg)  \d s'\d s\,.  \label{eqn.ARAR}
\end{align}
Since
\begin{equation*}
\Delta X_{s, t_{j+1}\wedge t}=X_{t_{j+1}\wedge t}- X_{s}=\int_s^{t_{j+1}\wedge t} b(X_u)\d u+\Delta B_{s,t_{j+1}\wedge t},
\end{equation*}
we can rewrite   \eqref{eqn.ARAR}   as    follows
\begin{align} \label{eqn.exact ARAR}
 I_{j, j'} =
&\int_{t_{j}}^{t_{j+1}\wedge t}\int_{t_{j'}}^{t_{j'+1}\wedge t}
 \EE\left[ A_{j}^{t}b_{1}(s, j)\lc \int_s^{t_{j+1}\wedge t} b(X_u)\d u+\Delta B_{s,t_{j+1}\wedge t}  \rc  \right.\nonumber  \\
&\qquad\qquad \left.\cdot A_{j'}^{t} b_{1}(s', j')\lc \int_{s'}^{t_{j'+1}\wedge t}  b(X_v)\d v+\Delta B_{s', t_{j'+1}\wedge t}  \rc \right] \d s'\d s \nonumber\\
=& I_1+I_2+I_3+I_4,
\end{align}
where
\begin{align}
I_1:=& \int_{t_{j}}^{t_{j+1}\wedge t}\int_{t_{j'}}^{t_{j'+1}\wedge t}\int_s^{t_{j+1}\wedge t}\int_{s'}^{t_{j'+1}\wedge t}
 \EE \bigg( A_{j}^{t}b_{1}(s,j)b(X_u)\nonumber\\
 &\qquad \cdot A_{j'}^{t}b_{1}(s',j')b(X_v)\bigg) \d u\d v\d s'\d s ; \nonumber\\
I_2:=& \int_{t_{j}}^{t_{j+1}\wedge t}\int_{t_{j'}}^{t_{j'+1}\wedge t}\int_s^{t_{j+1}\wedge t}
 \EE\left[ A_{j}^{t}b_{1}(s, j) b(X_u)\cdot A_{j'}^{t} b_{1}(s', j')\Delta B_{s', t_{j'+1}\wedge t}  \right] \d u \d s'\d s; \nonumber \\
I_3:=& \int_{t_{j}}^{t_{j+1}\wedge t}\int_{t_{j'}}^{t_{j'+1}\wedge t}\int_{s'}^{t_{j'+1}\wedge t}\EE\left[ A_{j}^{t}b_{1}(s,j)\Delta B_{s,t_{j+1}\wedge t}  \cdot A_{j'}^{t} b_{1}(s', j') b(X_v)\right] \d v \d s'\d s ; \nonumber\\
I_4: =&
 \int_{t_{j}}^{t_{j+1}\wedge t}\int_{t_{j'}}^{t_{j'+1}\wedge t}
 \EE\lc A_{j}^{t}b_{1}(s,j)\Delta B_{s,t_{j+1}\wedge t}  \cdot A_{j'}^{t}b_{1}(s', j')\Delta B_{s',t_{j'+1}\wedge t}  \rc \d s'\d s . \nonumber
\end{align}
In the above   we have omitted the dependence of   $I_1, I_2, I_3, I_4$   on $j, j'$ for  the sake of simplicity.

\medskip
 \noindent\emph{Step 4: Estimate of $I_{1}$.}
In the following we derive the upper-bound for $I_{1}$,$\cdots$, $I_{4}$.
First, let us consider $I_1$. From the polynomial  growth assumptions on $b$, Proposition \ref{p.3.2}, \eqref{eqn.Ajt}, and Lemma \ref{Malliavin b1}, we obtain that
\[
\EE \bigg| A_{j}^{t}b_{1}(s,j)b(X_u)
  \cdot A_{j'}^{t}b_{1}(s',j')b(X_v)\bigg|\le C<\infty\,.
\]
Thus
\begin{align}\label{eqn.I1}
|I_1|\le  C\int_{t_{j}}^{t_{j+1}\wedge t}\int_{t_{j'}}^{t_{j'+1}\wedge t}\int_s^{t_{j+1}\wedge t}\int_{s'}^{t_{j'+1}\wedge t}  \d u\d v\d s'\d s \le C\De_{j}^2\De_{j'}^2\,.
\end{align}

\smallskip
 \noindent\emph{Step 5: Estimate of $I_{2}$ and $I_{3}$.}
We turn to  $I_2$. From
  \eqref{eqn.Ajt}, \eqref{e.4.25}-\eqref{e.4.26},  Lemma \ref{Malliavin b1}, the polynomial growth assumption
on $b$, Proposition \ref{p.3.2} and Proposition \ref{p.3.6}, we see easily (by using the product rule of Malliavin derivative) that
\[
\EE\sup_{s,s',t,u,j,j', \xi} \left|D_{\xi} \lc A_{j}^{t}b_{1}(s, j) b(X_u)\cdot A_{j'}^{t} b_{1}(s', j')\rc\right|<\infty.
\]
Thus by Lemma \ref{lemma.FDe_B}, we have
\[
\left|\EE\left[ A_{j}^{t}b_{1}(s, j) b(X_u)\cdot A_{j'}^{t} b_{1}(s', j')\Delta B_{s', t_{j'+1}\wedge t}  \right]\right|\le C \De_{j'}\,.
\]
Consequently, we have
\begin{equation}
|I_2|\le   C\De_{j'}\int_{t_{j}}^{t_{j+1}\wedge t}\int_{t_{j'}}^{t_{j'+1}\wedge t}\int_s^{t_{j+1}\wedge t}\d u \d s'\d s \le C \De_{j}^2\De_{j'}^2\,.
\end{equation}
Exactly in  the same way  we have
\begin{equation}
|I_3|\le    C \De_{j}^2\De_{j'}^2\,.
\end{equation}

\smallskip
 \noindent\emph{Step 6: Estimate of $I_{4}$.}
Denote
\[
F(s,s',t, j, j'):= A_{j}^{t}b_{1}(s,j)    A_{j'}^{t}b_{1}(s', j')\,.
\]
Similar to the above argument we can show that
\begin{eqnarray*}
&&\EE \bigg[\sup_{  s,s',t, j, j'\, \xi, \zeta}\big( |F(s,s't, j, j')|
+|D_\xi F(s,s',t, j, j')|\\
&&\qquad\quad +|D_{\xi \zeta}^2F(s,s',t, j, j')|\big)\bigg]\le C<\infty\,.
\end{eqnarray*}
From Lemma \ref{lemma.FDe_B} it follows then
\begin{align} \label{eqn.I4}
\begin{split}
|I_4|
\le& \int_{t_{j}}^{t_{j+1}\wedge t}\int_{t_{j'}}^{t_{j'+1}\wedge t}\left|
 \EE\lc F(s,s',t, j, j') \Delta B_{s,t_{j+1}\wedge t}   \Delta B_{s',t_{j'+1}\wedge t}  \rc \right| \d s'\d s \\
\le& C  \int_{t_{j}}^{t_{j+1}\wedge t}\int_{t_{j'}}^{t_{j'+1}\wedge t}  \left[
\langle \mathbf{1}_{[s,t_{j+1}\wedge t]},  \mathbf{1}_{[s',t_{j'+1}\wedge t]} \rangle  + |t_{j+1}-s| \cdot |t_{j+1}-s'|\right]  \d s'\d s\\
\le&   C  \De_j \De_{j'} \langle \mathbf{1}_{[t_{j},t_{j+1}\wedge t]},  \mathbf{1}_{[t_{j'},t_{j'+1}\wedge t]} \rangle +C  \De_j^2\De_{j'}^2.
\end{split}
\end{align}

\smallskip
 \noindent\emph{Step 7: Conclusion.}
Plugging \eqref{eqn.I1}-\eqref{eqn.I4} into
\eqref{eqn.exact ARAR},  we conclude that
\begin{eqnarray*}
\EE(|Z_{t}|^{2}) \le \sum_{j,j'=0}^{ {k}} |I_{j,j'}|
\le  C\sum_{j,j'=0}^{ {k}} \De_j^2\De_{j'}^2
+
C \De_j \De_{j'}  \sum_{j,j'=0}^{ {k}} \langle \mathbf{1}_{[t_{j},t_{j+1}\wedge t]},  \mathbf{1}_{[t_{j'},t_{j'+1}\wedge t]} \rangle
\\ \le C |\pi|^{2}\,.
\end{eqnarray*}
 This proves the theorem.
\end{proof}

\section{ Limiting distribution}\label{limit}

In the section we consider the asymptotic error distribution of the backward Euler scheme.

\subsection{A linear ordinary differential equation}

Let $\phi_t: [0,T]\to \RR^{m\times m}$ be the solution to   the linear differential  equation
\begin{eqnarray}\label{eqn.phi_t}
\phi_{t} = \text{Id} + \int_{0}^{t} \partial b (X_{s} ) \phi_{s} \d s.
\end{eqnarray}
We have the following estimate for \eqref{eqn.phi_t}.
\begin{lemma}\label{l.5.1}
Let $\phi_t $ 
be defined by \eqref{eqn.phi_t}.

 For all $s,t: 0\leq  s<t\leq T$ we have relations
\begin{eqnarray}
  |\phi_{t}    |  \leq e^{ \lambda (t-s) } |\phi_{s}   |, \quad
\qquad \text{and} \qquad
\sup_{t\in [0,T]} |\phi_{t}' | \leq \kappa \sup_{t\in [0,T]}(1 + |X_{t}|^{\mu}) \cdot    \sup_{t\in [0,T]}  |\phi_{t} |.\label{e.5.3}
\end{eqnarray}
\end{lemma}
\begin{proof}
Take $s,t: 0\leq s<t\leq T$.
Let  $\phi_{i, t} $ be the $i$-th  column  of $\phi_{t}$.  Using \eqref{eqn.phi_t} we can write   that
\begin{align*}
   |\phi_{i,t} |^{2} = & |\phi_{i,s} |^{2} + 2\int_{s}^{t}  \langle \partial b (X_{u} ) \phi_{i,u}, \phi_{i,u} \rangle \d u
 .
\end{align*}
As in the proof of Lemma \ref{corollary.I-tJ} we can show that $\partial b (X_{u} ) $ satisfies   condition \eqref{eqn.lam} with $\lambda=\kappa$. Therefore, we get
\begin{eqnarray*}
   |\phi_{i,t} |^{2} &\leq &  |\phi_{i,s} |^{2}  + 2\lambda \int_{s}^{t}   |\phi_{i,u}  |^{2}  \d u .
\end{eqnarray*}

  Gronwall's inequality then  yields
\begin{eqnarray*}
   |\phi_{i,t}  |^2 \leq e^{ 2\lambda (t-s)} |\phi_{i,s}  |^2  .
\end{eqnarray*}
This gives the first inequality in \eqref{e.5.3}.
\end{proof}
Now we want to approximate the above equation
\eqref{eqn.phi_t}  by
  the  following ``forward-backward'' Euler scheme:
\begin{eqnarray}\label{eqn.fbeuler1}
\phi_{t }^{\pi} = \phi_{t_{k}}^{\pi} +   \partial b (X_{t_{k}} ) \phi_{t }^{\pi} \, \cdot \, (t-t_k) \, ,
\qquad
\phi^{\pi}_{0} = \text{Id}
.
\end{eqnarray}
In particular, when $t=t_{k+1}$  we have
\begin{eqnarray}\label{eqn.fbeuler2}
\phi_{t_{k+1}}^{\pi} = \phi_{t_{k}}^{\pi} +   \partial b (X_{t_{k}} ) \phi_{t_{k+1}}^{\pi} \Delta_{k}\, ,
\qquad
\phi^{\pi}_{0} = \text{Id}
.
\end{eqnarray}
Note that solving    \eqref{eqn.fbeuler1} for $\phi^{\pi}_{t}$  we have
\begin{align}
\phi_{t}^{\pi}
&=
( I - \partial b (X_{t_{k}} )  (t-t_{k}) )^{-1}   \phi_{t_{k}}^{\pi}\nonumber
\\
& =  \tilde A_{t} \cdot
\phi^{\pi}_{t_k} ,
\label{e.5.2.1}
\end{align}
where for $t_{\ell}<t\leq t_{\ell+1}$ and $\ell=0,\dots,n-1$ we define
\begin{eqnarray}\label{eqn.Atilde}
\tilde A_{t} =( I - \partial b (X_{t_{\ell}} ) (t-t_{\ell}) )^{-1}.
\end{eqnarray}

\begin{prop}\label{p.5.2}
Let $\phi_t$ and $\phi_t^{\pi}$ be defined  by  \eqref{eqn.phi_t} and \eqref{eqn.fbeuler1}, for any $p>1$, we have the following $L^{p}$-convergence
\begin{eqnarray*}
\sup_{t\in[0,T]}  \| \phi^{\pi}_{t} - \phi_{t} \|_{p}    \leq C |\pi|^{H}.
\end{eqnarray*}
\end{prop}
\begin{proof}
On the subinterval $ t\in (t_k, t_{k+1}]$,
we   write
\begin{align}
\label{eqn.phi_t2}
\phi_t
 &=  \phi_{t_k} +\int_{t_k}^t \partial b(X_s) \phi_s \d s \nonumber\\
&=  \phi_{t_k} +  \partial b(X_{t_k} ) \phi_t\cdot (t-{t_k}) +\tilde{R}_{t} \,,
\end{align}
where for $t_{k}<t\leq t_{k+1}$ we denote
\[
\tilde{R}_{t} =\int_{t_k}^t \partial b(X_s) \phi_s \d s-\int_{t_k}^t  \partial b(X_{t_k} ) \phi_t \d s\,.
\]
Thus, we have  
\begin{align}
\phi_{t}
=&
( I - \partial b (X_{t_{k}} )  (t-t_{k}) )^{-1}
\lc
 \phi_{t_{k}}  +\tilde{R}_{t}
\rc
= \tilde{A}_{t} \phi_{t_{k}}  +\tilde{A}_{t} \tilde{R}_{t}
\,. \label{e.5.6}
\end{align}

Denote by $z_t=\phi_t-\phi_t^{\pi}$ and $z_k=\phi_{t_k}-\phi_{t_k}^{\pi}$. Then,  similar to \eqref{eqn.ztk}, by taking difference between \eqref{e.5.2.1} and \eqref{e.5.6} we obtain
\begin{align}\label{eqn.zt1}
z_t=\tilde{A}_t z_k + \tilde{A}_t \tilde{R}_{t} \,.
\end{align}

   Computing  \eqref{eqn.zt1} recursively lead to
\begin{align}\label{eqn.zt2}
z_t= \tilde{A}_{t}\tilde{R}_{t} +\tilde{A}_{t}\tilde{A}_{t_{k}}\tilde{R}_{t_{k}} +\cdots+\tilde{A}_{t}\tilde{A}_{t_{k}}\cdots \tilde{A}_{t_{1}}  \tilde{R}_{t_{1}}\,
. 
\end{align}
Rewrite $\tilde{R}_{t}$ as
\begin{eqnarray}\label{eqn.rtilde}
\tilde{R}_{t} =
\int_{t_k}^t \lc\partial b(X_s)- \partial b(X_{t_k} )\rc \phi_s \d s+\int_{t_k}^t  \partial b(X_{t_k} ) \lc\phi_s - \phi_{t} \rc\d s \,.
\end{eqnarray}
Now apply  Assumption \ref{ass-1} (A2),   relation \eqref{eqn.xholder} and   relations in \eqref{e.5.3}   to the first integral in \eqref{eqn.rtilde}, and then apply Assumption \ref{ass-1} (A2)  and relation \eqref{e.5.3} again to the second integral.  We obtain
\begin{align}\label{eqn.Rtilden}
  ||\tilde{R}_t|| \le C |\pi |^{1+H}  (1+\sup_{0\le s\le T} |X_s|^{ \mu}+\|B\|_{\al})^{2} \,.
\end{align}

Let us recall the Lemma \ref{corollary.I-tJ}. We take $\theta=1$ and assume that $\kappa\cdot  |\pi|<1$.  Then we have
the following estimation for $\tilde A_{t}$
\begin{eqnarray}\label{eqn.Atilden}
||\tilde A_{t}|| \le (1- \kappa(t-t_{\ell}))^{-1}\le e^{2\kappa(t-t_{\ell})}.
\end{eqnarray}
Substituting   the estimations \eqref{eqn.Rtilden} and \eqref{eqn.Atilden} into \eqref{eqn.zt2}, it has
\begin{align*}
||z_t||\le  C |\pi |^{H}  (1+\sup_{0\le s\le T} |X_s|^{ \mu}+\|B\|_{\al})^{2}\,,
\end{align*}
where $C$ is a constant with respect to $\kappa, \mu, H, T$, but independent of the partition $\pi$. According to the boundness of $X_t$, see \eqref{e.3.4}, we obtain earlier proof the proposition.
\end{proof}

Take $t_{k}<t\leq t_{k+1}$, and $t_j\le t_k$ and denote $\phi_{t_j}(t)=\phi_t\phi_{t_j}^{-1}$. According to the definition \eqref{eqn.phi_t} of $\phi$ we have
\begin{equation}
\phi_{t_j}(t)=\text{Id} +\int_{t_j}^t  \partial b(X_s) \phi_{t_j}(s) \d s\,,\quad t_j\le t\le T\,.  \label{e.5.7}
\end{equation}
On the other hand, using    the expression \eqref{eqn.fbeuler1}  we get for
$\phi_{t_j}^\pi (t):=\phi_t^\pi (\phi_{t_j}^\pi)^{-1}$
\begin{equation} \label{e.5.8}
\phi_{t_j}^\pi (t)= \text{Id} +  \int_{t_j}^t  \partial b(X_{t_k}) \phi_{t_j}^{\pi}(s) \d s\,,\quad t_j\le t_k< t\leq t_{k+1}\le T\, .
\end{equation}
Therefore, in the similar way as for Proposition \ref{p.5.2} we can show the following.

\begin{prop}\label{p.5.3}
Let $\phi_{t_j}(t)$ and $\phi_{t_j}^\pi (t)$ be defined by \eqref{e.5.7} and \eqref{e.5.8} separately.  Then for any $p>1$, we have the following $L^{p}$-convergence
\begin{eqnarray*}
\sup_{0\le t_j\le t, t\in[0,T]}  \| \phi_{t_j}^\pi (t)  - \phi_{t_j}(t)\|_{p}    \leq C |\pi|^{H}.
\end{eqnarray*}
\end{prop}

\subsection{The asymptotic error of the backward Euler scheme.}

In this subsection we prove our main result on  the asymptotic error  of the backward Euler scheme. The proof is based on  the representation \eqref{eqn.zt} of the error process $Z_t$ and some approximation results for $R^{t}_{j}$ and $A^{t}_{j}$.

We first prove  the following lemma.

\begin{lemma}\label{l.5.4}
Recall that $R_{k} =
 R_{k}(t_{k+1})=\int_{t_k}^{t_{k+1}} b(X_s)-b(X_{t_{k+1}})\d s$ is defined by  \eqref{eqn.rkt}. Let
\begin{eqnarray}\label{eqn.rt}
\hat{R}_{k} =R_k +R_{1k} + R_{2k}
 \,,
\end{eqnarray}
where we define
\begin{equation}\label{eqn.r1r2}
R_{1k}:= \partial bb (X_{t_{k}}) \int_{t_{k}}^{t_{k+1}} ({t_{k+1}} -  s) \d s \quad
\text{and} \quad
R_{2k}:=\partial b (X_{t_{k}}) \int_{t_{k}}^{t_{k+1}}   (B_{t_{k+1}} - B_{s}) \d s,
\end{equation}
and we denote
\begin{eqnarray}\label{eqn.pbb}
\partial b  b (x) = \sum_{i=1}^{m}\frac{\partial b}{\partial x^{i}} (x) b^{i}(x) .
\end{eqnarray}
Then  for any $\al<H$ we have
\begin{eqnarray}\label{eqn.rk.r}
\left|  \hat{R}_{k}
 \right| \leq G_\al  \Delta_{k}^{2\al+1},
\end{eqnarray}
where  $G_\al $ is some random variable  independent of $n$ and $k$  such that $\|G_\al\|_{p}<\infty$ for any $p>0$.
\end{lemma}
\begin{proof}
Using   It\^o formula and \eqref{maineq} we can write
 \begin{align}\label{eqn.rex1}
b(X_{t_{k+1}})-b(X_s) &=
\int_{s}^{t_{k+1}}   \partial b  (X_{u})   \d X_{u}
\nonumber
\\
&=
\int_{s}^{t_{k+1}}   \partial b      b(X_{u})\d {u} + \int_{s}^{t_{k+1}}   \partial b  (X_{u})    \d B_{u}.
\end{align}
Let us denote
\begin{eqnarray}\label{eqn.rk}
r_{k} (s)= \int_{s}^{t_{k+1}}  \partial b      b(X_{u}) -  \partial b      b(X_{t_{k}}) \d  u + \int_{s}^{t_{k+1}}   \partial b  (X_{u})   - \partial b  (X_{t_{k}}) \d B_{u} .
\end{eqnarray}
Then using \eqref{eqn.rex1} it is clear that
 \begin{eqnarray}\label{eqn.rex}
b(X_{t_{k+1}})-b(X_s)  &=& \partial bb (X_{t_{k}}) (  {t_{k+1}} -  s ) +  \partial b (X_{t_{k}})  (B_{t_{k+1}} - B_{s}) +r_{k}(s)
.
\end{eqnarray}
 Integrate  both  sides of equation \eqref{eqn.rex} and recall the definition of $\hat{R}_{k} $ in  \eqref{eqn.rt}. We get
\begin{eqnarray}\label{eqn.rt2}
\hat{R}_{k} = -\int_{t_{k}}^{t_{k+1}}  r_{k}(s) ds .
\end{eqnarray}

In order to show \eqref{eqn.rk.r} it suffices to prove a bound for    $r_{k}(s)$.
We apply     Lemma \ref{lem.young} to the second integral of the right side of \eqref{eqn.rk}  to get
for any  $\al\in (1/2,H)$
\begin{eqnarray*}
|r_{k}(s)| \leq \|\partial bb (X)\|_{\al} \Delta_{k}^{1+\al } + \|\partial b (X)\|_{\al} \|B\|_{\al} \Delta_{k}^{2\al } .
\end{eqnarray*}
Thanks to    Assumption \ref{ass-1} (A2) and  Lemma \ref{p.3.3} for  $X$, we obtain  that  $\|\partial bb (X)\|_{\al}$ and $\|\partial b (X)\|_{\al} \|B\|_{\al}$ are bounded by some  random variable $G_\al$ which has any finite moment.
That is, we have   the estimate
\begin{eqnarray*}
 \sup_{s\in [t_{k}, t_{k+1}]}  | r_{k}(s) |
 \le  G_\al \Delta_{k}^{2\al  }.
\end{eqnarray*}
 Applying the above estimate to \eqref{eqn.rt2} we obtain   \eqref{eqn.rk.r}.
\end{proof}

Let us turn to the estimate of   $A^{t}_{j}$. We first have   the following lemma for $A_{k}$.
\begin{lemma}
Let  $A_{k} = A_{t_{k}}$, $k=0,1,\dots, n$ be defined in \eqref{eqn.At}. Take $\al \in (1/2, H)$.  The following relation holds almost surely
\begin{eqnarray}\label{eqn.aexpb}
\sup_{k=0,1,\dots,n}\Delta_{k}^{-1-\al} \| A_{k+1}  -  \exp\big(      \partial b(X_{t_k}) \Delta_{k} \big) \| <\infty.
\end{eqnarray}
\end{lemma}
\begin{proof}
We first use the elementary matrix  identity $A-B=A(B^{-1}-A^{-1})B$ to  write
\begin{eqnarray}\label{eqn.akd}
A_{k+1}  -  \exp\big(      \partial b(X_{t_k}) \Delta_{k} \big) = A_{k+1} \cdot \Big(  \exp\big(   -   \partial b(X_{t_k}) \Delta_{k} \big) -A_{k+1}^{-1}   \Big) \cdot \exp\big(      \partial b(X_{t_k}) \Delta_{k} \big) .
\end{eqnarray}
According to  Lemma \ref{At}   $A_{k+1}$     is uniformly bounded in $k$ by a constant. On the other hand,  using Assumption \ref{ass-1} (A2)     and    Lemma \ref{p.3.3}    we can  show  that $\sup_{k}\exp\big(      \partial b(X_{t_k}) \Delta_{k} \big) \to 1$ almost surely. So in order to bound \eqref{eqn.akd} it    suffices  to show that   the second term on the right side of \eqref{eqn.akd} multiplied by $\Delta_{k}^{-1-\al}$ is uniformly bounded in $k$.

Observe that by    applying the  Taylor expansion $e^{A} = I+A+e^{\theta A}\cdot A^{2}/2$ to $  \exp\big(   -   \partial b(X_{t_k}) \Delta_{k} \big)$ and using  the definition    \eqref{eqn.At} of $A$   we can write
\begin{align*}
&
 \exp\big(   -   \partial b(X_{t_k}) \Delta_{k} \big) - A_{k+1}^{-1}
\\
 =& I -   \partial b(X_{t_k}) \Delta_{k} +  \exp\big(   -  \theta\cdot \partial b(X_{t_k}) \Delta_{k} \big)\cdot
 ( \partial b(X_{t_k}) \Delta_{k})^{2}
 \\
 &
 \qquad-
 \left( I-\Delta_{k}\int_0^1 \partial b(uX_{t_{k+1}}+(1-u)Y_{t_{k+1}})\d u\right)
 ,
\end{align*}
where $\theta$ is some constant in $(0,1)$. Rearrange the above we get
\begin{eqnarray*}
\Delta_{k}^{-1-\al}\Big(
 \exp\big(   -   \partial b(X_{t_k}) \Delta_{k} \big) - A_{k+1}^{-1}
  \Big)
  &=&I_{1}+I_{2},
\end{eqnarray*}
where
\begin{align*}
I_{1} =&
  \Delta_{k}^{-\al} \Big(
 -  \partial b(X_{t_k})  +
  \int_0^1 \partial b(uX_{t_{k+1}}+(1-u)Y_{t_{k+1}})\d u
  \Big),
  \\
 I_{2} =&
   \Delta_{k}^{1-\al} \exp\big(   -  \theta\cdot \partial b(X_{t_k}) \Delta_{k} \big)
 \cdot ( \partial b(X_{t_k})  )^{2}.
\end{align*}

In the following we show that both $I_{1}$ and $I_{2}$ are uniformly bounded in $k$, which, together with the discussion after equation \eqref{eqn.akd}, shows that \eqref{eqn.aexpb} holds.

We first note that $I_{2}$ can be bounded in the
similar way as for  $\exp\big(      \partial b(X_{t_k}) \Delta_{k} \big) $ in \eqref{eqn.akd}, and we have $\sup_{k}I_{2}<\infty$ almost surely. For the estimate of  $I_{1}$, we observe that by the mean value theorem and Assumption \ref{ass-1} (A2)  for $\partial b$ we have
\begin{eqnarray*}
I_{1} \leq C \Delta_{k}^{-\al}
(1+|X|^{\mu}+|Y|^{\mu} ) ( |X_{t_{k}}-X_{t_{k+1}}| + | X_{t_{k+1}} - Y_{t_{k+1}} | ) .
\end{eqnarray*}

  We combine the boundedness of $X_t$ and $Y_t$ in \eqref{e.3.4} and \eqref{eqn.ynb}, the first relation in \eqref{eqn.xholder}, and the  strong convergence result \eqref{eqn.errorT}. We conclude  $I_1\le C $.
The proof is complete.
\end{proof}

We are ready to derive the following approximation for $A^{t}_{j}$.

\begin{lemma} \label{lam.Aphi}
For any $p\in [1, 2)$ and any $\al<H$,  we have
\begin{eqnarray}\label{eqn.lam.Aphi}
\EE \sup_{0\le t_j<t\le T} \left| A_{j}^{t} - \phi_{t}^\pi \left( \phi_{{t_j}}^\pi\right)^{-1} \right|^p \le C_{p,\al}  |\pi|^{p\al}.
\end{eqnarray}
\end{lemma}
\begin{proof}
By iterating \eqref{e.5.2.1}, we   observe that
\begin{eqnarray}
\phi^{\pi}_{t} =\tilde{A}_{t}\prod_{\ell=j+1}^{k} \tilde A_{ {\ell}} \cdot
\phi^{\pi}_{t_j}  .\label{e.5.2}
\end{eqnarray}
So multiplying $(\phi^{\pi}_{t_{j}})^{-1}$ on both sides of  \eqref{e.5.2} leads to
 \begin{eqnarray*}
 \phi^{\pi}_{t } (\phi^{\pi}_{t_{j}})^{-1} 
 = \tilde{A}_{t}\prod_{\ell=j+1}^{k } \tilde A_\ell  \, .
\end{eqnarray*}

On the other hand, recall the expression of $A_j^t$ (see \eqref{eqn.At} and \eqref{eqn.atj})
\begin{eqnarray*}
A_{j}^{t}
&=& A_t \lc \prod_{\ell ={j+1}}^{k } A_\ell \rc =
A_t \prod_{\ell ={j+1}}^{k} \left(I-\Delta_{\ell-1} \int_0^1 \partial b(uX_{t_\ell}+(1-u) Y_{t_{\ell}} ) \d u\right)^{-1} .
\end{eqnarray*}
Thus, applying the estimates  \eqref{eqn.bound At2} and \eqref{eqn.Atilden} of $A_{t}$ and $\tilde{A}_{t}$ we get
\begin{align}
& \|A_{j}^{t}-\phi_{t }^\pi(\phi_{t_j})^{-1}\|
\leq \|A_t-\tilde A_{t } \|\prod_{\ell =j+1}^{k } \|A_{\ell}\|
\nonumber\\
   & +\|\tilde{A}_{t}\|\sum_{i=j+1}^k   \|A_{j+1}\| \cdot \|A_{j+2}\| \cdots \|A_{i-1}\| \cdot
\|A_i-\tilde A_i\| \cdot \|\tilde{A}_{i+1}\| \cdots \|
\tilde A_k\|
\nonumber
\\
&\qquad
\leq
C (k-j)\sup_{t\in [0,T]}\|A_{t}-\tilde{A}_{t}\|
\,.\label{e.5.10}
\end{align}
It follows from the definition of  $A_t$ in \eqref{eqn.At} and $\tilde A_t$ in \eqref{eqn.Atilde} Assumption \ref{ass-1} (A2) that
\begin{align*}
\|A_t-\tilde A_t\|
&\le  \De_{i-1}  \|A_t \|\left\| \int_0^1 \partial b(uX_{t}+(1-u) Y_{t }) \d  u-\partial b( X_{\eta(t)})
\right\| \|\tilde A_t\|\\
&\le  C  \De_{i-1}  \sup_{0\le t\le T}  \left[1+  | X_t|^\mu +|Y_t|^\mu\right] \sup_{0< t\le T}
(|X_t-Y_t|+|X_{t}-X_{\eta(t)}|) \,.
\end{align*}
Applying   \eqref{e.3.4},  \eqref{eqn.ynb}, the convergence result \eqref{eqn.errorT} and the first relation in \eqref{eqn.xholder}  we obtain
\begin{eqnarray}\label{eqn.aa}
\EE\sup_{t\in [0,T]}\|A_t-\tilde A_t\| \leq C|\pi|^{\al+1}
\end{eqnarray}
for any $\al<H$.
Finally, substituting  \eqref{eqn.aa} into \eqref{e.5.10} we conclude the estimate  \eqref{eqn.lam.Aphi}.   The proof is now complete.
\end{proof}
Combine the Lemma \ref{lam.Aphi} with Proposition \ref{p.5.3}, we obtain the following lemma.
\begin{lemma}\label{l.5.7}
 For any $p\in [1, 2)$,  we have
\begin{eqnarray*}
\EE \sup_{0\le t_j<t\le T} \left| A_{j}^{t} - \phi_{t}  \left( \phi_{{t_j}} \right)^{-1} \right|^p \le C_p  |\pi|^{p\al}.
\end{eqnarray*}
\end{lemma}
Now we state the main theorem of this section.
\begin{theorem}  For any $p\in [1, 2)$,  we have the following convergence in $L_{p}$:
\begin{equation}
nZ_t\to U_t:=\frac12 \int_0^t \phi_t\phi_s^{-1} \partial b b(X_s) \d s+\frac 12 \int_0^t \phi_t\phi_s^{-1} \partial b  (X_s) \d B_s
 \,.
\label{eqn_Uform}
\end{equation}
Moreover, $U_t$ satisfies
\begin{equation}
dU_t=\partial b(X_t) U_t \d t+
\frac12  \partial b b(X_t) \d t+\frac 12  \partial b  (X_t) \d B_t\,. \label{eqn_Ueqn}
\end{equation}
\end{theorem}
\begin{proof}
 Let $k$ be such that  $t_{k}<t\leq t_{k+1}$. Recall                                                                                                                                                                                                                                                                                                                                                                                                                                                                          relation \eqref{eqn.zt}. We have
\begin{align*}
nZ_{t} &= \sum_{j=0}^{ k -1} A_{j}^t  (nR_{j}  ) + A_{k }^t  (nR_{k }^{t} ).
\end{align*}
Recall the representation  \eqref{eqn.rt} for $R_{k}$. We consider the following decomposition of   $nZ_{t}$
\begin{align}
nZ_{t}
 &= \sum_{j=0}^{ k -1} \lc A_{j}^{t} -\phi_t\phi_{t_j}^{-1} \rc(nR_{j}  )
 +\sum_{j=0}^{k  -1}  \phi_t\phi_{t_j}^{-1}     n\hat R_{j}   - \sum_{j=0}^{k -1}  \phi_t\phi_{t_j}^{-1}     n  R_{1j}
\nonumber  \\
 &\qquad -\sum_{j=0}^{k -1}  \phi_t\phi_{t_j}^{-1}     n R_{2j} +  A_{k }^t  (nR_{k }^t  )
 \nonumber
 \\
 &=:I_1+I_2+I_3+I_4+I_5 \,.
 \label{eqn.nz}
\end{align}
In the following we consider the convergence of $I_{i}$, $i=1,\dots, 5$ individually.

We first consider the convergence of $I_{5}$. Note  that by applying    Assumption (A2) and the H\"{o}lder continuity  of $X_t$ in  \eqref{eqn.xholder} to \eqref{eqn.rkt} we get
\begin{equation}
\EE|R_k(t)|^p\le C_p |\pi|^{p(1+H)},
\qquad p\geq 1
\,.
\label{e.5.16}
\end{equation}
So, taking into account      the estimate  of $A_j^t$ in  \eqref{eqn.Ajt}, we have
\begin{equation}
\EE |I_5|^p = \EE \left| A_{k}^t  (nR_{k}^t  ) \right|^p <C|\pi|^{pH}\,.
\end{equation}
This implies the convergence $I_{5}\to0$ as $|\pi|\to0$.

For the convergence of $I_{1}$ we apply  Lemma \ref{l.5.7} and relation  \eqref{e.5.16}. This yields
\begin{equation}
\EE |I_1|^p \le \sum_{j=0}^{ k-1} \EE  \left| \lc A_{j}^{t} -\phi_t\phi_{t_j}^{-1} \rc(nR_{j} ) \right|^p \le C \sum_{j=0}^{ k_2-1}|\pi|^{(\al+H)p} \leq  |\pi|^{(\al+H)p-1}  \,.
\end{equation}
Since $\al+H>1$ for $H>1/2$, we conclude that $I_{2}\to 0$ as $|\pi|\to 0$.

For $I_{2}$ we apply  Lemma  \ref{l.5.4}     to get
\begin{equation}
\EE |I_2|^p \le \sum_{j=0}^{k -1} \EE \left| \phi_t\phi_{t_j}^{-1} n\hat R_{j}\right|^p \le \sum_{j=0}^{k -1} |\pi|^{2\al p }\,.
\end{equation}
So  $  |I_2|\to 0$ as  $|\pi|\to 0$.

Let us consider the convergence of $I_{3}$ and $I_{4}$.
 According to the  definition of $R_{1j}$ we have
\begin{align}
I_3
&=\sum_{j=0}^{k-1}  \phi_t\phi_{t_j}^{-1}  \cdot  \partial b b(X_{t_j}) \cdot\frac{n}{2} (t_{j+1}-t_j)^2 \nonumber \\
&=\frac{1}{2n} \sum_{j=0}^{k-1}  \phi_t\phi_{t_j}^{-1}         \partial b b(X_{t_j}).
\end{align}
So  by the continuity of $\phi_t\phi_s^{-1} \partial b b(X_s)$ in $L_{p}$  we obtain $$I_3 \to \frac12 \int_0^t \phi_t\phi_s^{-1} \partial b b(X_s) \d s$$  in $L_{p}$ as $|\pi|\to 0$.

For $I_{4}$ we write
\begin{align}
I_4
&=\sum_{j=0}^{k_2-1}  n\phi_t\phi_{t_j}^{-1}    \partial b(X_{t_j}) \int_{t_j}^{t_{j+1}} \int_s^{t_{j+1}} \d B_u \d s\nonumber \\
&=n\sum_{j=0}^{k_2-1}  \phi_t\phi_{t_j}^{-1}    \partial b(X_{t_j}) \int_{t_j}^{t_{j+1}} (u-t_j)   \d B_u .
\end{align}
We apply   \cite[Corollary 7.2]{hu2016rate}   to get the convergence in $L_{p}$ as $|\pi|\to 0$:
 $$I_{4}\to\frac 12 \int_0^t \phi_t\phi_s^{-1} \partial b  (X_s) \d B_s . $$
In summary of the convergence of $I_{i}$, $i=1,\dots, 5$,
  we conclude the convergence \eqref{eqn_Uform}.  It is straightforward to check that $U_t$ satisfies
\eqref{eqn_Ueqn}.
\end{proof}

\section{Numerical Experiments}\label{Experiments}

In this section, we validate  our theoretical results     by performing two numerical experiments. In the following    $B_t $, $t\geq 0$ denotes  a standard   fractional Brownian motion with Hurst parameter $H\in (0,1)$.

\begin{example}\label{example1}
We consider the following one-sided Lipschitz stochastic differential  equation
\begin{align}\label{eqn.x^3}
\d X_t = -X_t^3 \d t +  \d B_t,
\end{align}
with  $X(0)=5$ and Hurst parameter $H=0.6$.
In this example we  compare the performance of three numerical methods, namely the backward Euler-Maruyama method,  the forward Euler-Maruyama method  and the Crank-Nicolson method.

We   use  the backward Euler method    with step size $ 0.0001$ to compute a sample of  the exact solution. For the simulation of the  three  numerical methods
we consider    two  step sizes:  $0.02$ and $0.08$.
 Table \ref{table.new} and \ref{table.new1} show the results for the sampling in  two different step sizes.  The  exact solution and the three  numerical approximations of equation \ref{eqn.x^3} are  evaluated at  $T=0.1,0.2,\cdots,0.8$.
 Clearly, the step size $0.08$ is too large for the Euler-Maruyama method and the Crank-Nicolson method on this one-sided Lipschitz problem. On the other hand, when  the step size is cut to $0.02$  all three methods perform well.  
\begin{table*}[htb]
	\footnotesize
	\centering
 \setlength{\tabcolsep}{3pt}
	\caption{The values of three numerical methods and exact solution for \eqref{eqn.x^3} with step size $0.08$.}\label{table.new}
	\begin{tabular}{cccccccccc}
		\toprule
		{value} & {$T=0.08$} & {$T=0.16$} & {$T=0.24$}
		& {$T=0.32$}  &{$T=0.40$} & {$T=0.48$} & {$T=0.56$}
		& {$T=0.64$} &{$T=0.72$} \\		
		
		\midrule
		
		EM  &-5.0498 &4.8234 &-11.5869 &45.4107 &-4.2063e+5 &4.4809e+6 &-8.0082e+46

            &4.1562e+127  & -Inf \\

	     CN &-0.7427 &3.7617 &-4.6843 &23.6835 &-2.1032e+5 &2.2404e+6 &-4.0041e+46
            &2.0781e+127   & -Inf \\

        BEM &3.5643  &2.7000 &2.2182 &1.9562 &1.6969 &1.5096 &1.4973
            &1.7401  &  1.5886 \\
		
	   exact&2.3221  &1.7058 &1.6053 &1.6447 &1.3966 &1.2202 &1.3487
            &1.6903    & 1.4840   \\

		\bottomrule		
	\end{tabular}	
\end{table*}

\begin{table*}[htb]
	\footnotesize
	\centering
	\caption{The values of three numerical methods and exact solution for \eqref{eqn.x^3} with step size $0.02$.}\label{table.new1}
	\begin{tabular}{cccccccccc}
		\toprule
		{value} & {$T=0.08$} & {$T=0.16$} & {$T=0.24$}
		& {$T=0.32$}  &{$T=0.40$} & {$T=0.48$} & {$T=0.56$}
		& {$T=0.64$} &{$T=0.72$} \\		
		
		\midrule
		
		EM  &1.8854 &1.5087 &1.4626 &1.5672 &1.3547 &1.1569 &1.3245
            &1.6991  & 1.4537 \\

	     CN &2.2632 &1.7030 &1.5849 &1.6349 &1.4030 &1.2139 &1.3472
            &1.7000   & 1.4752 \\

        BEM &2.6409 &1.8973 &1.7072 &1.7025 &1.4514 &1.2709 &1.3698
            &1.7009  &  1.4967 \\
		
	   exact&2.3221 &1.7058 &1.6053 &1.6447 &1.3966 &1.2202 &1.3487
            &1.6903    & 1.4840   \\

		\bottomrule		
	\end{tabular}	
\end{table*}
\end{example}

\begin{example}\label{example2}

Let   $B_{t} = (B_t^{1}, B_{t}^{2})$ be a  two-dimensional     fractional Brownian motion with Hurst parameter $H\in (0,1)$ and assume that  $B_t^{1}$ and $B_t^{2}$ are mutually independent.
In this example we consider  the two-dimensional stochastic differential equation studied  in  \cite{abbaszadeh2013design}:
\begin{equation}
\label{eg2}
\left\{ \begin{split}
\d X_t &= \lc X_t-Y_t- X_t^3-X_tY_t^2 \rc \d t+ \d B_t^{1}, \\
\d Y_t &= \lc X_t+Y_t- X_t^2Y_t- Y_t^3 \rc \d t+ \d B_t^{2},
\end{split}\right.
\end{equation}
with the initial value $X(0)=1,Y(0)=1$ for   $t\in [0,1]$. It is easy to verify that     equation  \eqref{eg2} satisfies the one-sided Lipschitz condition.
\end{example}
\begin{figure}[htb]
\centering
\subfigure[$H=0.6$]{
\includegraphics[width=6cm]{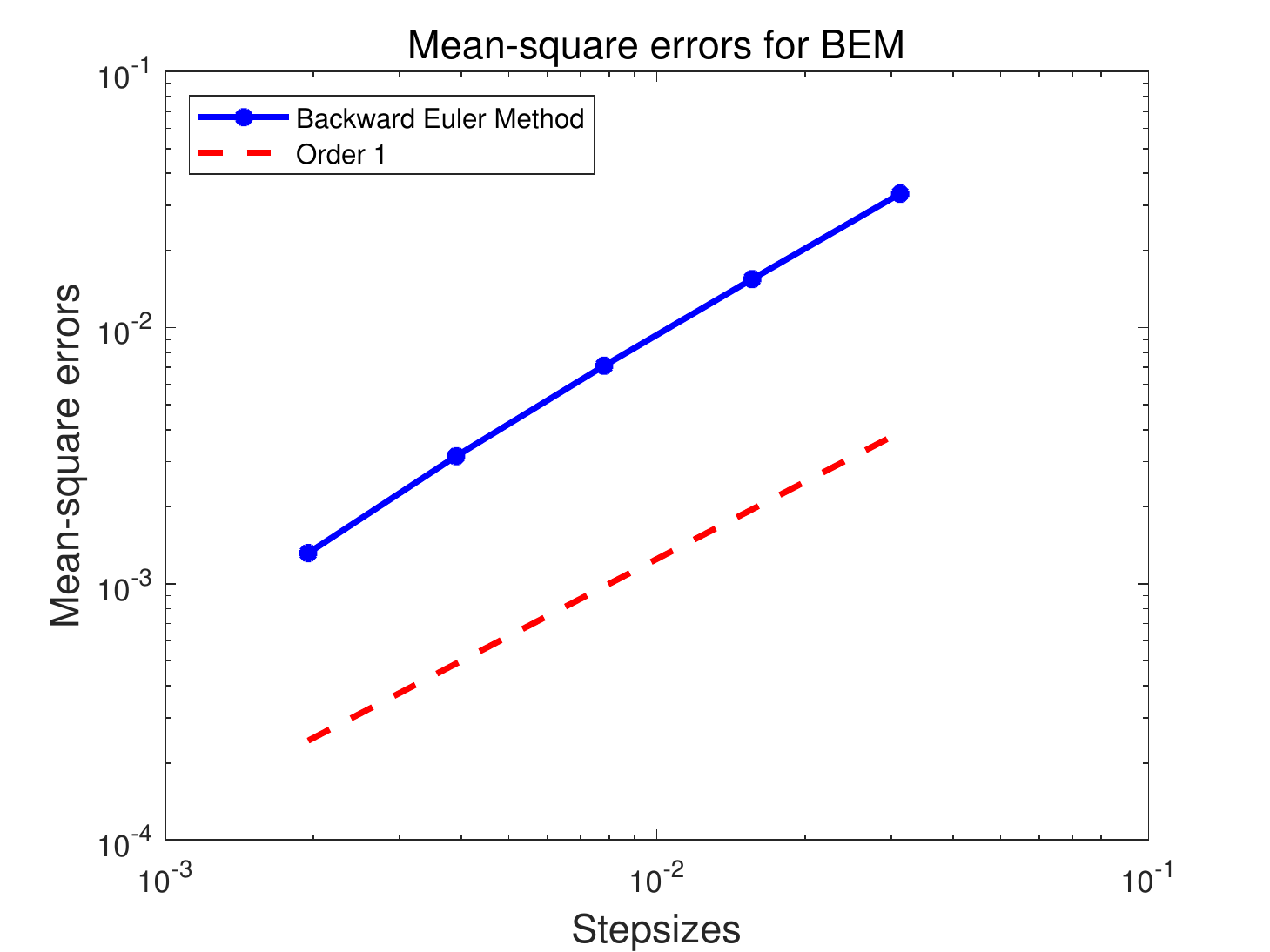}
}
\subfigure[$H=0.7$]{
\includegraphics[width=6cm]{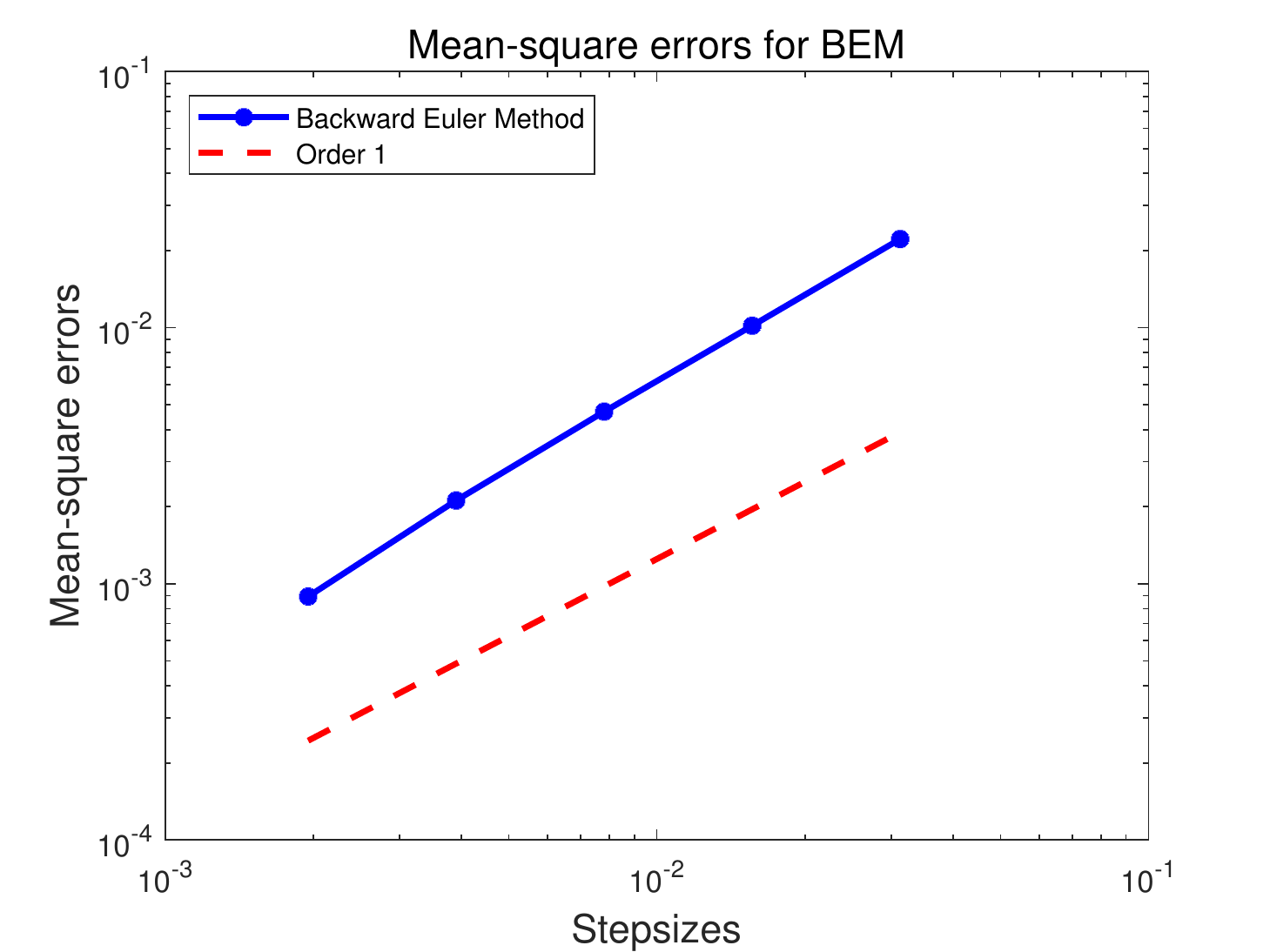}
}
\subfigure[$H=0.8$]{
\includegraphics[width=6cm]{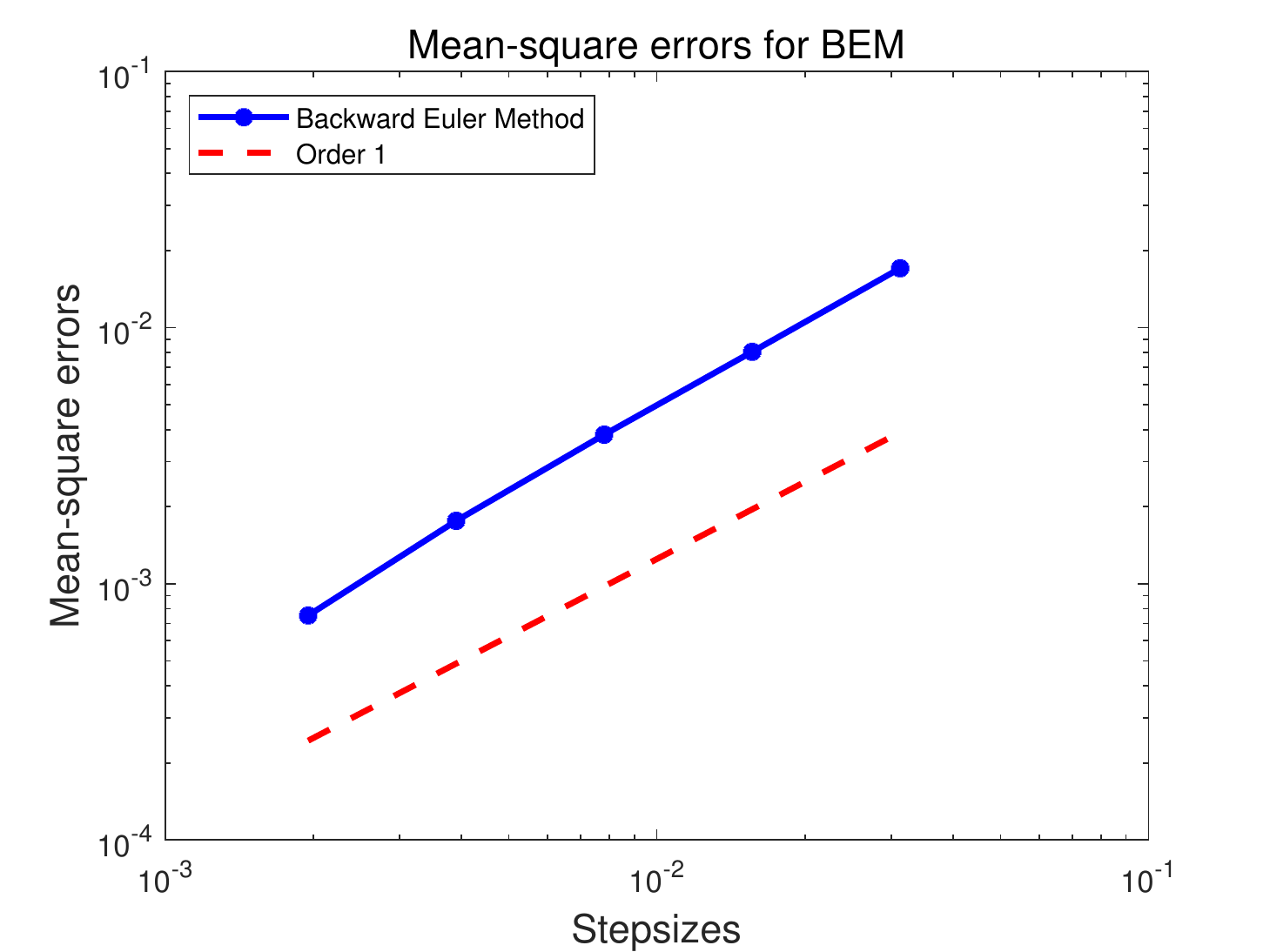}
}
\subfigure[$H=0.9$]{
\includegraphics[width=6cm]{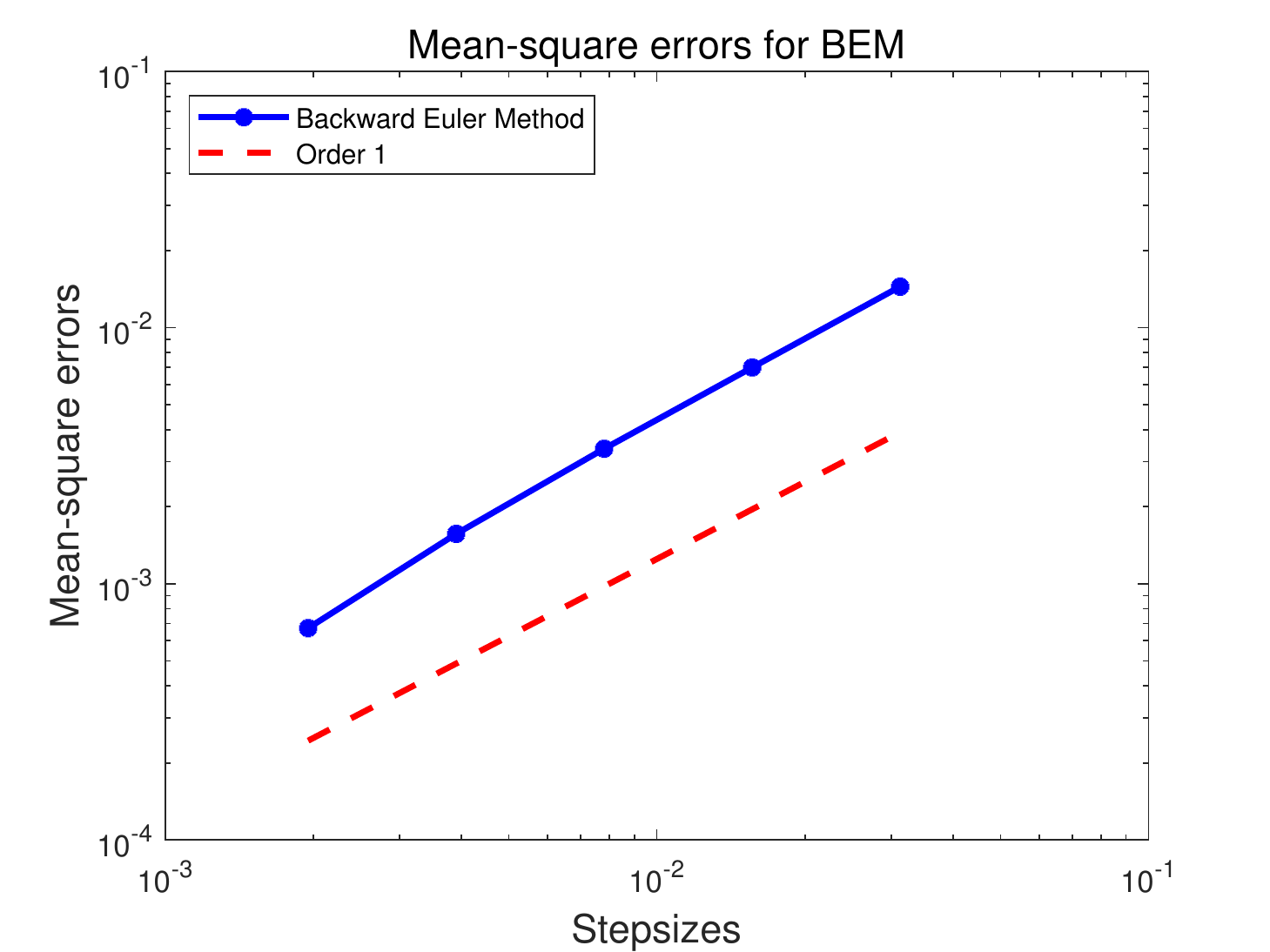}
}
\caption{The strong mean-square convergence error of approximation of SDEs (\ref{eg2}).}
\label{eg2-1}
\end{figure}
 To find  the convergence rate of the backward Euler method,
 in our numerical computation we take the time step sizes  $2^{-k}$, $k=5,6,\cdots,9$.
As in the previous Example \ref{example1}    the numerical solution with a much smaller  step size (in this example we take the step size $2^{-11}$) is used to replace  the exact solution.
We perform  $M=1000$ simulations, and we compute  the  mean square error  by
 \begin{align*}
 \epsilon (T)=\sqrt{\frac{1}{M}\sum_{\ell =1}^M(|X^\ell_T-X^\ell _N|^2+|Y^\ell_T-Y^\ell _N|^2)}.
 \end{align*}

For the four subfigures in Figure \ref{eg2-1}, the red dotted reference line has a slope of $1$. When $H=0.6$, the slope of the blue line is $1.1616$ through least squares fitting the error, and when $H=0.9$, the slope is $1.1016$, the rest of the blue lines have slopes between these two values.

Accurate numerical values of error and convergence orders under different step sizes are listed in Table \ref{eg2-2}. Thus we numerically verify that  our theoretical results are in accordance with numerical results.

\begin{table*}[htb]
	\footnotesize
	\centering
	\caption{Mean squared of errors and convergence rates of BEM method for SDEs~(\ref{eg2}).}\label{eg2-2}
	\begin{tabular}{cccccccccc}
		\toprule
		{$\Delta t$} & \multicolumn{2}{c}{$H=0.6$} & \multicolumn{2}{c}{$H=0.7$} & \multicolumn{2}{c}{$H=0.8$}
		& \multicolumn{2}{c}{$H=0.9$}  \\		
		\cmidrule(r){2-3} \cmidrule(r){4-5} \cmidrule(r){6-7} \cmidrule(r){8-9} 			
		&  error      &  order
		
		&  error      &  order
		
		&  error      &  order
		
		&  error      &  order  \\
		
		\midrule
		
		$1/{2^5}$ &3.3335e-2&-     &2.2177e-2& -
                  &1.7053e-2&-     &1.4462e-2& -       \\
		
		$1/{2^6}$ &1.5464e-2&1.1081&1.0176e-2&1.1239
                  &8.0469e-3&1.0835&6.9912e-3&1.0487    \\
		
		$1/{2^7}$ &7.1006e-3&1.1229&4.7004e-3&1.1143
                  &3.8213e-3&1.0744&3.3680e-3&1.0536    \\
		
	   $1/{2^{8}}$&3.1505e-3&1.1724&2.1154e-3&1.1519
                  &1.7612e-3&1.1175&1.5667e-3&1.1042    \\
		
	   $1/{2^{9}}$&1.3181e-3&1.2571&8.9192e-4&1.2459
                  &7.5221e-4&1.2274&6.7140e-4&1.2225    \\		
		\bottomrule		
	\end{tabular}
\label{ta1}	
\end{table*}

\section{Appendix}

In this section we consider the differential equation:
\begin{eqnarray}\label{eqn.osl}
x_{t} = x_{t_{0}} +\int_{t_{0}}^{t} b(x_{s}) \d s,
\qquad t \in [t_{0}, \infty) \subset \RR_{+} ,
\end{eqnarray}
where
 we assume that

(i) The function $b$ is one-sided Lipschitz: there is a constant $\kappa$ such that
\begin{eqnarray*}
\langle x-y, b(x) - b(y) \rangle \leq \kappa |x-y|^{2},
\qquad \forall x, y \in \RR^{m} \text{ and } t\geq 0.
\end{eqnarray*}

(ii) Both $b$ and $\partial b$ are continuous on $\RR_{+}\times\RR^{m}$. Here $\partial b$ denotes the Jacobian $\frac{\partial b(x)}{\partial x}$ of $b$.

\begin{example}
It is easy to show that the following two equations are examples of  \eqref{eqn.osl} and the   conditions (i)-(ii) are satisfied:

(i) The equation
 \begin{eqnarray*}
y_{t}'  = b(y_{t}+g_{t} ) ,
\end{eqnarray*}
where $b$ is one-sided Lipschitz and $g$ is a continuous function.

(ii) The linear equation
\begin{eqnarray*}
x_{t}' = U_{t} x_{t} ,
\end{eqnarray*}
where the eigenvalues of $U_{t}$ are uniformly bounded for all $t\geq 0$.
\end{example}
The following existence and uniqueness results about equation \eqref{eqn.osl} are mostly  well-known. We   did not find the exact explicit   result   for \eqref{eqn.osl} in the literature  and   so  we include   a proof for  the sake of  completeness.
\begin{proposition}\label{prop.ode}
There exists a unique solution $x$ to  equation \eqref{eqn.osl}, and the solution satisfies the   relation:
\begin{eqnarray*}
|x_{t} |^{2} \leq \left(
|x_{t_{0}}|^{2}+ \int_{t_{0}}^{t} |b(0)|^{2}ds
\right) \cdot e^{(2\kappa +1)(t-t_{0})}.
\end{eqnarray*}
\end{proposition}

\begin{proof}
We start by defining
\begin{eqnarray*}
\tau = \sup\big\{ t: \text{there exists a solution to equation \eqref{eqn.osl} on } [t_{0}, t]  \big\}.
\end{eqnarray*}
We first  note that by Peano's theorem  it is easy to see that $\tau>t_{0}$.
On the other hand,     the one-sided Lipschitz condition     implies that for any $t<\tau$ there exists a unique solution on the interval $[t_{0}, t]$ (see e.g. \cite[Lemma 12.1]{hairer1996solving}).

In the following we show that $\tau=\infty$ by contradiction, which then concludes the existence of the solution on $  [t_{0}, \infty)$.

We   calculate
\begin{align*}
|x_{t}|^{2} =&\, |x_{t_{0}}|^{2} +2 \int_{t_{0}}^{t} \langle x_{s}, b(x_{s})  \rangle \d s
\\
=&\, |x_{t_{0}}|^{2} +2 \int_{t_{0}}^{t} \langle x_{s}, b(x_{s}) - b(0)  \rangle \d s
-
2 \int_{t_{0}}^{t} \langle x_{s},   b(0)  \rangle \d s
\\
\leq &\, |x_{t_{0}}|^{2} + 2 \kappa \int_{t_{0}}^{t} |x_{s}|^{2}\d s + \int_{t_{0}}^{t} |x_{s}|^{2}ds + \int_{t_{0} }^{t} |b(0)|^{2}\d s  .
\end{align*}
Applying
Gronwall's inequality yields
\begin{align*}
|x_{t}|^{2} \leq \left(
|x_{t_{0}}|^{2}+ \int_{t_{0}}^{t} |b(0)|^{2}\d s
\right) e^{(2\kappa +1)(t-t_{0})}
\\
 \leq \left(
|x_{t_{0}}|^{2}+ \int_{t_{0}}^{\tau} |b(0)|^{2}\d s
\right) e^{(2\kappa +1)(\tau-t_{0})}.
\end{align*}
This implies that $b(x_{t})$ is bounded on $[t_{0}, \tau)$. Since $x_{t}' = b(t, x_{t})$ we obtain that $x$ is uniform continuous on the interval $[t_{0}, \tau)$. In particular,   the limit $\lim_{t\to \tau-}x_{t}$ exists. This shows that $x$ is a solution to equation \eqref{eqn.osl} on $[t_{0}, \tau]$.

Now consider the equation $x_{t}' = b(x_{t})$ with initial value $x_{\tau} =x_{\tau}$. Peano's theorem implies that the equation has a solution on the interval $[\tau, \tau+\ep]$ for some $\ep>0$. Combining the two functions $(x_{t}, t\in [t_{0}, \tau])$ and $(x_{t}, t\in [\tau,\tau +\ep])$ we obtain a solution to equation \eqref{eqn.osl} on the interval $[t_{0}, \tau+\ep]$. This contradicts to the definition of $\tau$. We conclude that $\tau = \infty$.       The proof is complete.
\end{proof}

\bibliographystyle{amsplain}
\bibliography{onesidedlipschitz}

\end{document}